\documentclass[11pt,letterpaper]{amsart}
\usepackage{graphicx}
\usepackage{hyperref}
\usepackage{tikz}
\usepackage{comment}
\usepackage{mathrsfs}
\usepackage[T1]{fontenc}

\usepackage{amsmath,amsthm,amssymb,amscd,enumerate}

\newtheorem{thm}{Theorem}[section]
\newtheorem{prop}[thm]{Proposition}
\newtheorem{lem}[thm]{Lemma}
\newtheorem{cor}[thm]{Corollary}

\theoremstyle{definition}
\newtheorem{definition}[thm]{Definition}

\theoremstyle{remark}

\newtheorem{remark}[thm]{Remark}

\numberwithin{equation}{section}

\allowdisplaybreaks[1]

\newcommand{\RR}{\mathbb{R}}
\newcommand{\ZZ}{\mathbb{Z}}
\newcommand{\NN}{\mathbb{N}}
\newcommand{\nn}{l}
\newcommand{\ppi}{q}
\newcommand{\qd}{\mathscr{A}}
\newcommand{\ee}{\epsilon}
\newcommand{\bb}{\mathfrak{b}_I}

\DeclareMathOperator{\Ker}{\mathrm{Ker}}

\newcommand{\hg}{\hat{g}}
\newcommand{\bG}{G}
\newcommand{\hG}{\hat{G}}
\newcommand{\fk}{\check{f}}
\newcommand{\gk}{\check{g}_{\alpha}}
\newcommand{\hk}{\check{g}_{\beta}}

\newcommand{\Symp}{\mathrm{Symp}^c}
\newcommand{\Ham}{\mathrm{Ham}^c}
\newcommand{\Homeo}{\mathrm{Homeo}}
\newcommand{\Sympeo}{\mathrm{Sympeo}}
\newcommand{\Supp}{\mathrm{Supp}}
\newcommand{\flux}{\mathrm{Flux}}

\newcommand{\tSymp}{\widetilde{\mathrm{Symp}_0^c}}
\newcommand{\Cal}{\mathrm{Cal}}

\newcommand{\Rot}{\mathrm{Rot}}

\begin{document}

\title[Commuting symplectomorphisms]{Commuting symplectomorphisms on a surface and the flux homomorphism}

\author[M. Kawasaki]{Morimichi Kawasaki}
\address[Morimichi Kawasaki]{Department of Mathematical Sciences, Aoyama Gakuin University, 5-10-1 Fuchinobe, Chuo-ku, Sagamihara-shi, Kanagawa, 252-5258, Japan}
\email{kawasaki@math.aoyama.ac.jp}

\author[M. Kimura]{Mitsuaki Kimura}
\address[Mitsuaki Kimura]{Department of Mathematics, Kyoto University, Kitashirakawa Oiwake-cho, Sakyo-ku, Kyoto 606-8502, Japan}
\email{mkimura@math.kyoto-u.ac.jp}

\author[T. Matsushita]{Takahiro Matsushita}
\address[Takahiro Matsushita]{Department of Mathematical Sciences, University of the Ryukyus, Nishihara-cho, Okinawa 903-0213, Japan}
\email{mtst@sci.u-ryukyu.ac.jp}

\author[M. Mimura]{Masato Mimura}
\address[Masato Mimura]{Mathematical Institute, Tohoku University, 6-3, Aramaki Aza-Aoba, Aoba-ku, Sendai 9808578, Japan}
\email{m.masato.mimura.m@tohoku.ac.jp}

\makeatletter
\@namedef{subjclassname@2020}{%
\textup{2020} Mathematics Subject Classification}
\makeatother



\begin{abstract}

%



%

Let $(S,\omega)$ be a closed connected  oriented surface whose genus $l$ is at least two  equipped with a symplectic form.
Then we show the vanishing of the cup product of the fluxes of commuting symplectomorphisms. This result may be regarded as an obstruction for commuting symplectomorphisms. In particular, the image of an abelian subgroup of $\Symp_0(S, \omega)$ under the flux homomorphism is isotropic with respect to the natural intersection form on $H^1(S;\RR)$.
The key to the proof is  a refinement of  the non-extendability result,  previously given  by the first-named and second-named authors, for Py's Calabi quasimorphism $\mu_P$ on $\Ham(S, \omega)$.



\end{abstract}

\maketitle

\tableofcontents

\section{Introduction}\label{intro section}

\subsection{Main theorem}\label{subsec=mainthm}

Let $(M, \omega)$ be a connected symplectic manifold, and $\Symp(M, \omega)$ the group of symplectomorphisms with compact supports on $M$.
Let $\Symp_0(M, \omega)$ denote the identity component of $\Symp_0(M, \omega)$, and $\widetilde{\Symp_0}(M, \omega)$ its universal cover.
The {\it flux homomorphism} $\widetilde{\flux}_\omega \colon \widetilde{\Symp_0}(M, \omega) \to H_c^1(M; \RR)$ is defined by
\[\widetilde{\flux}_\omega([\{\psi^t\}_{t\in[0,1]}])=\int_0^1[\iota_{X_t}\omega]dt,\]
where  $X_t$ is the vector field generating the symplectic isotopy $\{\psi^t\}$.
The image of $\pi_1(\Symp_0(M, \omega))$ with respect to $\widetilde{\flux_\omega}$ is called the {\it flux group of $(M, \omega)$}, and the flux homomorphism descends  to a homomorphism
\[\flux_\omega \colon \Symp_0(M, \omega) \to H^1(M ; \RR) / \Gamma_\omega, \]
which is also called the flux homomorphism. The flux homomorphism is a fundamental object in symplectic geometry and theory of diffeomorphism groups, and has been extensively studied by many authors
such as \cite{Ban}, \cite{LMP}, \cite{Ke}, \cite{O}, \cite{Buh}, and \cite{KLM}.

Let $S$ be a closed orientable surface whose genus $l$ is at least two, and $\omega$ a volume form on $S$. In this case, since $\pi_1(\Symp_0(S , \omega))$ is trivial  (see Section~7.2 of \cite{P01} for example), the flux homomorphism gives a group homomorphism from $\Symp_0(S, \omega)$ to $H^1(S ; \RR)$.

The goal in this paper is to study the commuting elements in $\Symp_0(S, \omega)$ from the viewpoint of the flux homomorphism.
The main result in this paper is the following:

\begin{thm}[Main theorem, Theorem~$\ref{thm=main}$] \label{main thm}
Let $S$ be a closed orientable surface whose genus $l$ is at least two and  $\omega$ a symplectic form on $S$.
Assume that a pair $f,g \in \Symp_0(S,\omega)$ satisfies $fg = gf$. Then
\[ \flux_\omega(f) \smile \flux_\omega(g) = 0\]
holds true. Here, $\smile \colon H^1(S;\RR) \times H^1(S;\RR) \to H^2(S;\RR) \cong \RR$ denotes the cup product.
\end{thm}

In other words, for a pair, $f$ and $g$, of commuting elements in $\Symp_0(S,\omega)$, the intersection number $\bb (\flux_{\omega}(f),\flux_{\omega}(g))$ equals zero. Here, $\bb (\cdot, \cdot)$ denotes the intersection form on $H^1(S;\RR)$. See the beginning of Section~\ref{section=proof_mainthm} for more details.

Theorem \ref{main thm} states the relationship between commuting two elements in $\Symp_0(S,\omega)$ and the cup product of the images of these elements by the flux homomorphism for a closed surface $S$ of genus at least $2$. Results of this type date back to the following theorem of Rousseau, which treats the case of \emph{open} symplectic manifolds.
\begin{thm}[Rousseau {\cite[Proposition 4.1]{Rou}}] \label{rousseau}
Let $(M,\omega)$ be a $2n$-dimensional open symplectic manifold.
Then, for every $f,g\in \tSymp(M,\omega)$,
\[\Cal([f,g]) = n \cdot \widetilde{\flux}_\omega(f) \smile \widetilde{\flux}_\omega(g) \smile [\omega]^{n-1}.\]
Here, $\Cal$ is the Calabi homomorphism \textup{(}see Subsection $\ref{subsec=Calabi}$\textup{)} and we identify  $H_c^{2n}(M;\RR)$  with $\RR$.
\end{thm}
If $n=1$, then Theorem \ref{rousseau} implies a  result similar to  Theorem \ref{main thm} on open surfaces. We will discuss in Subsection~\ref{subsec=sketch} (Corollary~\ref{cor=open}).

We note that one of the keys to Theorem \ref{rousseau} is the existence of the well-defined Calabi homomorphism; this breaks down for the case where $M$ is \emph{closed}. 
For this reason, our proof of Theorem~\ref{main thm} is quite different from those of Theorem~\ref{rousseau} by Rousseau (and Corollary~\ref{cor=open}), as we will sketch the strategy in Subsection~\ref{subsec=sketch}.

 By combining Theorem \ref{main thm} with a theorem of Franks and Handel \cite{FH}, we obtain the following corollary:


\begin{cor}[Restriction on commuting symplectomorphisms in terms of the flux homomorphism]\label{Zn action}
Let $S$ be a closed connected orientable surface whose genus $\nn$ is at least two and $\omega$ a symplectic form on $S$.
Then, for every   nilpotent subgroup $A$ of $\Symp_0(S,\omega)$, the inequality.
\[ \dim_\RR \langle \flux_\omega (A) \rangle_\RR \le l.\]
holds true. Here, for a subset $T$ of a vector space over $\RR$, we write $\langle T \rangle_\RR$ to mean the $\RR$-linear subspace spanned by $T$.
\end{cor}

 We note that Corollary~$\ref{Zn action}$ in particular applies to the case where $A$ is abelian.

\begin{proof}[Proof of Corollary~$\ref{Zn action}$ modulo Theorem~$\ref{main thm}$]
 Here we employ the result of Franks and Handel \cite[Corollary 1.10]{FH}, stating that
every nilpotent subgroup of $\Symp_0(S,\omega)$ is abelian.
Consider $H^1(S ; \RR)$ as a symplectic vector space whose symplectic form is provided by the intersection form. Set $V_A = \langle \flux_\omega (A) \rangle_\RR$. Then,
Theorem~\ref{main thm} implies that $V_A$ is isotropic subspace of $H^1(S ; \RR)$, and this means $\dim V_A \le l = \dim_\RR H^1(S ; \RR) / 2$.
\end{proof}

\begin{remark}
It is  not difficult
to see that there exist $l$ elements $f_1, \cdots, f_l$ of $\Symp_0( S, \omega)$ such that $\flux_{\omega}(f_1), \cdots, \flux_{\omega} (f_l)$ are linearly independent (see Remark~\ref{remark tightness 1}).
Hence, the inequality in Corollary~\ref{Zn action} is tight.
\end{remark}

\begin{remark}
 Although for every abelian subgroup $A$ of $\Symp_0(S,\omega)$ the image of $A$ under $\flux_\omega$ is contained in some $l$-dimensional vector space, it is  not difficult to see that the rank of the image as a $\ZZ$-module can be large. To see this, let $X$ be a vector field on $S$ such that $\mathcal{L}_X\omega=0$ and $[\iota_X\omega]$ is a non-trivial element of $H^1(S;\RR)$.
Let $\{\varphi_X^t \}_{t\in\RR}$ denote the flow generated by $X$.
Let $\Phi \colon \ZZ^{\oplus \NN} \to\Symp_0(S,\omega)$ be the homomorphism defined by
\[ \Phi(k_1, k_2, \ldots,)=\varphi_X^{ \sum_{i=1}^\infty k_i \sqrt{p_i}},\]
where $p_i$ is the $i$-th smallest prime number. Here $\ZZ^{\oplus \NN}$ is the free abelian group of countably infinite rank. Then,  we can show  that $\flux_\omega \circ \Phi$ is injective.

\end{remark}

Note that Corollary~\ref{Zn action} immediately implies the following result, which was recently proved by the first and second authors.

\begin{cor}[Kawasaki--Kimura {\cite[Corollary 1.9]{KK}}] \label{cor Kawasaki-Kimura}
Let $S$ be a closed connected  orientable surface whose genus is at least two, and $\omega$ a symplectic form on $S$. Then, the flux homomorphism
\[
\flux_\omega \colon \Symp_0(S, \omega) \to H^1(S ; \RR)
\]
does \emph{not} have a section homomorphism.
\end{cor}

\subsection{Strategy of the proof}\label{subsec=sketch}

The proof of Theorem~\ref{main thm} is obtained by examining the extendability of certain quasimorphisms on the group of Hamiltonian diffeomorphisms $\Ham(S, \omega)$ on $(S, \omega)$.
Recall that a real-valued function $\mu \colon G \to \RR$ on a group $G$ is called a {\it quasimorphism} if there exists a non-negative real number $D \ge 0$ satisfying
\[
|\mu(xy) - \mu (x) - \mu(y)| \le D
\]
for every $x,y\in G$. The smallest  such $D$ is called the {\it defect} of $\mu$, and denoted by $D(\mu)$. A quasimorphism $\mu$ is said to be {\it homogeneous} if $\mu(g^{ m}) = m \cdot \mu(g)$ for every $g \in G$ and for every $m \in \ZZ$.

Suppose that $G$ is a normal subgroup of another group $\hat{G}$, and that $\mu$ is a homogeneous quasimorphism on $G$. We say that $\mu$ is {\it $\hat{G}$-invariant} if
\[
\mu(gxg^{-1}) = \mu(x)
\]
for every $g \in \hG$ and for every $x \in \bG$. We say that $\mu$ is {\it extendable to $\hat{G}$} if there exists a homogeneous quasimorphism $\hat{\mu}$ on $\hat{G}$ such that the restriction $\hat{\mu}|_{G}$ of $\hat{\mu}$ to $G$ coincides with $\mu$.  If a homogeneous quasimorphism on $G$ is extendable to $\hat{G}$, then it must be  $\hat{G}$-invariant; see Lemma~\ref{lem=inv}. See Subsection~\ref{subsec=qm} for more details on these concepts.

As is expected, there exists a $\hat{G}$-invariant homogeneous quasimorphism which is not extendable.
Our main investigation is the extendability of Py's Calabi quasimorphism $\mu_P$ (\cite{Py06}).
See Subsections~\ref{subsec=mu_P_const} and \ref{subsec=mu_P_property} for details on $\mu_P$.
The proof of Theorem \ref{main thm} is obtained by observing the extendability and non-extendability to certain subgroups of $\Symp_0(S, \omega)$ containing $\Ham (S, \omega)$.
Our main tool for the extendability of quasimorphisms is \cite[Proposition 1.6]{KKMM20}, which may be seen as an extension theorem for \emph{discrete groups}.
We will recall the statement in Proposition~\ref{prop=discrete} in Section~\ref{section=proof_mainthm}; it asserts that
if the short exact sequence
\[ 1 \to \bG \to \hG \xrightarrow{\ppi} Q \to 1\]
of groups \emph{virtually splits}, then every $\hG$-invariant homogeneous quasimorphism on $\bG$ is extendable to $\hG$.
Our main tool for proving the non-extendability is the following, which may be seen as a main part of Theorem \ref{main thm}. Here, $\langle \cdot \rangle$ denotes the subgroup generated by $\cdot$.
The positive integer $k_0$ appearing in the following theorem can be explicitly estimated in terms of $\bar{w}$ , $\nn$  and ${\rm area}(S)$; see Theorem~\ref{thm=non_ext} and Remark~\ref{remark=normalize}.

\begin{thm}[Non-extendability of Py's Calabi quasimorphism, see also Theorem~$\ref{thm=non_ext}$ for explicit $k_0$] \label{thm A}
Let $S$ be a closed connected orientable surface whose genus $\nn$ is at least two and $\omega$ a symplectic form on $S$.
Let $\bar{v}$, $\bar{w} \in H^1(S;\RR)$ with $\bar{v} \smile \bar{w}  \neq  0$. For a positive integer $k$, set $\Lambda_k = \langle \bar{v}, \bar{w} / k \rangle$ and $G_{\Lambda_k} = \flux_\omega^{-1}(\Lambda_k)$.
Then, there exists a positive integer $k_0$ such that $\mu_P$ is {\rm not} extendable to $G_{\Lambda_k}$ for every $k \ge k_0$.
\end{thm}

Theorem~\ref{thm A} in particular yields the following corollary.

\begin{cor}\label{strong Py nonex}
Let $S$ be a closed connected orientable surface whose genus $\nn$ is at least two, $\omega$ a symplectic form on $S$ and $V$ a linear subspace of $H^1(S;\RR)$.
If $\mathrm{dim}_\RR(V) > \nn$, then there does \emph{not} exist a homogeneous quasimorphism $\hat{\mu}$ on $\flux_\omega^{-1}(V)$ such that $\hat{\mu}|_{\Ham(S, \omega)} = \mu_P$.
\end{cor}

We note that Corollary \ref{strong Py nonex} was proved by the first and second authors when $V=H^1(S;\RR)$ \cite[Theorem 1.11]{KK}.

\begin{remark}
There exists an abelian subgroup $A$ of $\Symp_0(S,\omega)$ such that $\dim_{\RR}V_A=l$; see Remark~\ref{remark tightness 1}.
Hence, the estimation of the dimension of Corollary~\ref{strong Py nonex} is optimal.
\end{remark}

 As we mentioned in Subsection~\ref{subsec=mainthm}, Theorem~\ref{rousseau} by Rousseau immediately yields the following result on open symplectc surfaces.

\begin{cor}\label{cor=open}
Let $S$ be an \emph{open} symplectic surface and  $\omega$ a symplectic form on $S$.
Assume that a pair $f,g \in \tSymp(S,\omega)$ satisfies $fg = gf$. Then
\[ \widetilde{\flux}_\omega(f) \smile \widetilde{\flux}_\omega(g) = 0\]
holds true. 
\end{cor}

\begin{proof}
Since $fg=gf$, we have $[f,g]=\mathrm{id}$. Therefore, Theorem~\ref{rousseau} for $n=1$ implies that
\[
\widetilde{\flux}_\omega(f) \smile \widetilde{\flux}_\omega(g)=\Cal([f,g])=0,
\]
as desired.
\end{proof}
We compare our proof of Theorem~\ref{main thm} with that of Corollary~\ref{cor=open}. Corollary~\ref{cor=open} follows from Theorem \ref{rousseau}, and Rousseau's proof of Theorem \ref{rousseau} is based on elaborate calculation of vector fields. As we sketched in this subsection, our strategy of the proof of Theorem~\ref{main thm} is quite different from this. 
Nevertheless, the proof of Theorem \ref{main thm} has the following similarity to that of Corollary~\ref{cor=open}. 
For the proofs of Corollary~\ref{cor=open} (and Theorem \ref{rousseau}), the Calabi \emph{homomorphism} connects the commutator $[f,g]$ of $f,g\in \tSymp(S,\omega)$ and the cup product of these images by the flux homomorphism. 
In our proof of Theorem \ref{main thm}, we employ Py's Calabi \emph{quasimorphism} to connect these objects if $[f,g]=\mathrm{id}$. 
Here we note that $\Ham(S,\omega)$ is a simple group if $S$ is a closed surface \cite{Ban}  so that there does \emph{not} exist a non-zero homomorphism in our setting. 
It may be natural in our strategy to use Py's Calabi quasimorphism as a counterpart of the Calabi homomorphism (in the case of open symplectic manifolds) to prove Theorem~\ref{main thm}.

\begin{remark}\label{rem=fullgenRousseau}
A natural question might be to ask whether we can obtain a corresponding result of Theorem~\ref{rousseau} itself (not only Corollary~\ref{cor=open}) to the case of a closed symplectic manifold $(M,\omega)$. However, there does \emph{not} exist any non-zero homogeneous quasimorphism $\mu$ on $\widetilde\Ham(M,\omega)$ such that for every $f,g\in \tSymp(M,\omega)$,
\[
\mu([f,g])=\widetilde\flux_{\omega}(f)\smile \widetilde\flux_{\omega}(g) \smile [\omega]^{n-1}
\]
holds, where $\dim M = 2n$. Indeed, suppose that such $\mu$ exists. 
Then, since $\Ker(\widetilde\flux_{\omega})=\widetilde\Ham (M,\omega)$ (Propsoition \ref{survey on flux}), we must have for every $f,g\in \widetilde\Ham(M,\omega)$,
\[
\mu([f,g])=\widetilde\flux_{\omega}(f)\smile \widetilde\flux_{\omega}(g)\smile [\omega]^{n-1}=0.
\]
By Lemma~\ref{lem=Bavard}, this implies that $\mu$ is a genuine homomorphism on $\widetilde\Ham(M,\omega)$. However, it forces $\mu$ to be the zero-map since $\widetilde\Ham(M,\omega)$ is a simple group, a contradiction.

In this aspect, Theorem~\ref{main thm} might \emph{not} be a straightforward variant of Theorem~\ref{rousseau} (or Corollary~\ref{cor=open}).
\end{remark}


\subsection{$C^0$-version of the  Main Theorem}

We can formulate our main theorem in the case of symplectic homeomorphisms.
Recall that $\Sympeo(S, \omega)$ denotes the $C^0$-closure of $\Symp_0(S,\omega)$ in the group $\Homeo(S)$ of homeomorphisms of $S$. Let $\Sympeo_0(S,\omega)$ be the identity component of $\Sympeo(S,\omega)$, and let $\overline{\Ham(S,\omega)}^{C^0}$ denote the $C^0$-closure of $\Ham(S,\omega)$ in $\Homeo(S,\omega)$.
As we will explain in Section \ref{section=C^0}, when the genus of $S$ is at least two, Fathi \cite{Fa} showed that the flux $\flux_\omega \colon \Symp_0(S, \omega) \to H^1(S;\RR)$ can be extended to as a continuous homomorphism $\theta_\omega \colon \Sympeo_0(S,\omega) \to H^1(S;\RR)$, whose kernel is $\overline{\Ham(S,\omega)}^{C^0}$:
\[ 1 \to \overline{\Ham(S,\omega)}^{C^0} \to \Sympeo_0(S,\omega) \xrightarrow{\theta_\omega} H^1(S;\RR) \to 1.\]
See Proposition~\ref{survey on mfh} in Section~\ref{section=C^0} for more details.

The following theorem is the $C^0$-version of our main theorem.

\begin{thm}[Main theorem in the $C^0$ setting, Theorem~$\ref{thm=C^0}$]\label{C0 main thm}
Let $S$ be a closed connected orientable surface whose genus $\nn$ is at least two, and $\omega$ a symplectic form on $S$. Assume that a pair $f,g \in \Sympeo_0(S,\omega)$ satisfies $fg = gf$.
 Then,
 \[\theta_\omega(f) \smile \theta_\omega(g) = 0\]
 holds true.
\end{thm}

By the same argument of the proof of Corollary~\ref{Zn action}, we have the following:

\begin{cor}\label{C0 Zn action}
Let $S$ be a closed connected orientable surface whose genus $\nn$ is at least two and $\omega$ a symplectic form on $S$.
Then, for every abelian subgroup $A$ of $\Sympeo_0(S, \omega)$, the inequality
\[\mathrm{dim}_\RR \; \langle \theta_\omega (A) \rangle_\RR \le \nn\]
holds true.
\end{cor}

In the proof of Theorem \ref{C0 main thm}, we also investigate the extendability and non-extendability of a certain quasimorphism on $\overline{\Ham(S,\omega)}^{C^0}$. However, by the Calabi property, Py's quasimorphism $\mu_P$ is not continuous on $\Ham(S, \omega)$ in the $C^0$-topology; it is unclear whether $\mu_P$ admits an extension to $\overline{\Ham(S,\omega)}^{C^0}$.
To address this problem, we in addition consider a quasimorphism $\mu_B$ on $\Ham(S,\omega)$ constructed by Brandenbursky \cite{Bra}; see Subsections~\ref{subsec=mu_P_const} and \ref{subsec=mu_P_property} for details.

By a result of Entov, Polterovich, and Py \cite{EPP}, we have that $\mu_P - \mu_B$ are continuous in the $C^0$-topology; see Theorem~\ref{mfh property} and Corollary~\ref{cor=C^0-conti} in Subsection~\ref{subsec=conti}.


\subsection*{Organization of the paper}
Section~\ref{section=prelim} is for preliminaries: we review some concepts in symplectic geometry and basic properties of quasimorphisms. In addition, we sketch the constructions and properties of Py's Calabi quasimorphism $\mu_P$ and Brandenbursky's Calabi quasimorphism $\mu_B$. In Section~\ref{section=fluxlift}, we construct certain symplectomorphisms and compute their fluxes.
 Theorem~\ref{main thm} is established in Section~\ref{section=proof_mainthm}: Subsection~\ref{subsection=proof_non_ext}  is devoted to the proof of Theorem~\ref{thm=non_ext}, which is a precise version of the non-extendability result, Theorem~\ref{thm A}, for $\mu_P$. In the proof, the construction discussed in Section~\ref{section=fluxlift} plays a key role. In
 Subsection~\ref{subsection=proof_main}, we deduce Theorem~\ref{main thm} from Theorem~\ref{thm A} and Proposition~\ref{prop=discrete}. In Section~\ref{section=C^0}, we prove Theorem~\ref{C0 main thm}.


Throughout the present paper, for a pair of real numbers $a$ and $b$ and for a non-negative number $D$, we write $a \sim_D b$ to mean $|b - a| \le D$. We use $\NN$ for the set of positive integers.  For a group $G$, $e_G$ denotes the group unit of $G$.  For mutually commuting elements $\gamma_1 ,\dots, \gamma_{m}$ of a group, let $\prod_{j=1}^{m} \gamma_j$ denote the product of $\gamma_1 ,\dots, \gamma_{m}$.


\section{Preliminaries}\label{section=prelim}

\subsection{Symplectic geometry}\label{subsec=symp}
In this subsection, we review some concepts in symplectic geometry which we will need in the subsequent sections. For a more comprehensive introduction to this subject, we refer to \cite{Ban97, MS, P01}.

Let $(M,\omega)$ be a connected symplectic manifold.
Let $\Symp(M,\omega)$ denote the group of  symplectomorphisms of $(M,\omega)$ with compact support and $\Symp_0(M,\omega)$ denote the identity component of $\Symp(M,\omega)$.
In this section, we endow $\Symp(M,\omega)$ with the $C^\infty$-topology.

For a smooth function $H\colon M\to\mathbb{R}$, we define the \textit{Hamiltonian vector field} $X_H$ associated with $H$ by
\[\omega(X_H,V)=-dH(V)\text{ for every }V \in \mathcal{X}(M),\]
where $\mathcal{X}(M)$ is the set of smooth vector fields on $M$.

For a (time-dependent) smooth function $H\colon  [0,1] \times M\to\mathbb{R}$ with compact support and for $t \in  [0,1] $, we define a function $H_t\colon M\to\mathbb{R}$ by $H_t(x)=H(t,x)$.
Let $X_H^t$ denote the Hamiltonian vector field associated with $H_t$ and let $\{\varphi_H^t\}_{t\in\mathbb{R}}$ denote the isotopy generated by $X_H^t$ such that $\varphi^0=\mathrm{id}$.
We set $\varphi_H=\varphi_H^1$ and $\varphi_H$ is called the \emph{Hamiltonian diffeomorphism generated by $H$}.
For a connected symplectic manifold $(M,\omega)$, we define the group of Hamiltonian diffeomorphisms by
\[\Ham(M,\omega)=\{\varphi\in\mathrm{Diff}(M)\;|\;\exists H\in C^\infty(S^1\times M)\text{ such that }\varphi=\varphi_H\}.\]
Then, $\Ham(M,\omega)$ is a normal subgroup of $\Symp_0(M,\omega)$.

Let $\widetilde{\Symp_0}(M,\omega)$ denote the universal covering of $\Symp_0(M,\omega)$.
We define the (symplectic) flux homomorphism $\widetilde{\flux}_\omega\colon\widetilde{\Symp_0}(M,\omega)\to H_c^{1}(M;\RR)$ by
\[\widetilde{\flux}_\omega([\{\psi^t\}_{t\in[0,1]}])=\int_0^1[\iota_{X_t}\omega]dt,\]
where $\{\psi^t\}_{t\in[0,1]}$ is a path in $\Symp_0(M,\omega)$ with $\psi^0=1$ and $[\{\psi^t\}_{t\in[0,1]}]$ is the element of the universal covering $\widetilde{\Symp_0}(M,\omega)$ represented by the path $\{\psi^t\}_{t\in[0,1]}$.
It is known that $\widetilde{\flux}_\omega$ is a well-defined homomorphism.

We also define the (descended) flux homomorphism.
We set
\[\Gamma_\omega=\flux_\omega(\pi_1(\Symp_0(M,\omega))),\]
which is called the \textit{symplectic flux group}.
Then, $\widetilde{\flux}_\omega \colon \widetilde{\Symp_0}(M, \omega) \to H_c^{1}(M ; \RR)$ induces a homomorphism $\Symp_0(M,\omega)\to H_c^{1}(M;\RR)/\Gamma_\omega$, which is denoted by $\flux_\omega$.
\begin{prop}[\cite{Ban,Ban97,MS}]\label{survey on flux}
  Let $(M,\omega)$ be a closed connected symplectic manifold.
  Then, the following hold.
 \begin{enumerate}[$(1)$]
  \item The flux homomorphism $\widetilde{\flux}_\omega \colon \widetilde{\Symp_0}(M,\omega)\to H^1(M;\RR)$ is surjective. In particular, $\flux_\omega\colon\Symp_0(M,\omega)\to H^1(M;\RR)/\Gamma_\omega$ is surjective.
  \item The map $\widetilde\Ham(M,\omega) \to \widetilde{\Symp_0}(M,\omega)$ induced by the inclusion map $\Ham(M,\omega) \to \Symp_0(M,\omega)$ is injective.
  	In particular,  $\widetilde\Ham(M,\omega)$ can be seen as a subgroup of $\widetilde{\Symp_0}(M,\omega)$.
  \item $\mathrm{Ker}(\widetilde\flux_\omega)=\widetilde\Ham(M,\omega)$  and $\mathrm{Ker}(\flux_\omega)=\Ham(M,\omega)$.
\end{enumerate}
\end{prop}
If $S$ is a closed orientable surface whose genus is at least two, then  $\Symp_0( S,\omega)$ is simply  connected, and  in particular, the flux group $\Gamma_\omega$ of $(S, \omega)$  vanishes. Hence in this case, the flux homomorphism $\flux_\omega$ is a homomorphism from $\Symp_0(S, \omega)$ to $H^1(S ; \RR)$.

\subsection{Quasimorphisms}\label{subsec=qm}
First, we briefly recall the definition and basic properties of quasimorphisms. A real-valued function $\mu \colon G \to \RR$ on a group $G$ is called a {\it quasimorphism} if
\[
 D(\mu):= \sup_{x,y \in G}|\mu(xy) - \mu (x) - \mu(y)|
\]
 satisfies that $D(\mu)<\infty$.
The constant $D(\mu)$ is called the {\it defect} of $\mu$.
 By utilizing the symbol $\sim_D$ introduced at the end of Section \ref{intro section}, we have that
\[
\mu(xy)\sim_{D(\mu)} \mu(x)+\mu(y)
\]
for every $x$ and $y$ in $G$.
A quasimorphism $\phi$ is said to be {\it homogeneous} if
$\phi(g^m) = m \cdot \phi(g)$ for every $g \in G$ and for every $m \in \ZZ$.
The following properties are fundamental and well-known properties of homogeneous quasimorphisms. For the reader's convenience, we include the proof.

\begin{lem} \label{lem:qm}
  Let $\phi$ be a homogenous quasimorphism on a group $G$. Then, for every $x,y \in G$, the following hold true:
  \begin{enumerate}[$(1)$]
    \item $\phi(yxy^{-1})=\phi(x)$;
    \item if $xy=yx$, then $\phi(xy)=\phi(x)+\phi(y)$.
  \end{enumerate}
\end{lem}
\begin{proof}
For every positive integer $m$, we have that
\[  m \cdot \phi(y x y^{-1}) = \phi(y x^{m} y^{-1}) \sim_{2D(\phi)} \phi(y) + \phi(x^{m}) + \phi(y^{-1}) = {m} \cdot \phi(x).\]
This means that
\[ |\phi(yxy^{-1}) - \phi(x)| \le \frac{2D(\phi)}{m}\]
for every $m\in \NN$. By letting $m\to\infty$, we obtain (1).
Now suppose that $xy = yx$.
Then we have that
\[  m \cdot \phi(xy) = \phi(x^m y^m) \sim_{D(\phi)} \phi(x^m) + \phi(y^m) = m \cdot (\phi(x) + \phi(y) ) \]
for every $m\in \NN$.
This means that
\[ |\phi(xy) - \phi(x) - \phi(y) | \le \frac{D(\phi)}{m}\]
for every $m\in \NN$.
Again, by letting $m\to \infty$, we verify (2).
\end{proof}

 The homogeneity condition on quasimorphisms is not restrictive in the following sense.

\begin{lem}[Homogenization of quasimorphisms]\label{lem=homog}
Let $G$  be a group and $\phi$ a quasimorphism on $G$, not necessarily homogeneous. Then there exists a quasimorphism $\phi_{\mathrm{h}}$ on $G$ that satisfies the following  three properties:
\begin{enumerate}[$(1)$]
  \item $\phi_{\mathrm{h}}$ is \emph{homogeneous};
  \item $\phi_{\mathrm{h}}(x)\sim_{D(\phi)}\phi(x)$ for every $x\in G$;
  \item $D(\phi_{\mathrm{h}})\leq 2D(\phi)$.
\end{enumerate}
\end{lem}

For each $\phi$,  we can show  that a quasimorphism $\phi_{\mathrm{h}}$ satisfying (1) and (2) is uniquely determined; the quasimorphism $\phi_{\mathrm{h}}$ is called the \emph{homogenization} of $\phi$.

\begin{proof}
First, we explain the construction of the homogenization $\phi_{\mathrm{h}}$ of $\phi$. Fix $x\in G$. Then  the sequence $(\phi(x^m)+D(\phi))_{m\in \NN}$ is subadditive.
By applying  Fekete's lemma,
we can define a map $\phi_{\mathrm{h}}\colon G\to \RR$ by
\[
\phi_{\mathrm{h}}\colon G\to \RR;\quad x\mapsto \lim_{m\to \infty}\frac{\phi(x^m)}{m}.
\]
 Note that $\phi(x^m) \sim_{(m-1)D(\phi)} m \phi(x)$ for every $x \in G$ and every $m \in \NN$.
Then, it is straightforward to verify (1), (2) and $D(\phi_{\mathrm{h}})\leq 4D(\phi)$. To obtain (3), see \cite[Lemma~2.58]{Ca}.
\end{proof}

Secondly, we recall the definition of quasi-invariance and invariance of  quasimorphisms for the pair $(\hG,\bG)$ of a group and its normal subgroup.

\begin{definition}[Invariant quasimorphism]\label{defn=inv_qm}
Let $\hG$ be  a group and $\bG$ a normal subgroup of $\hG$.
\begin{enumerate}[$(1)$]
 \item A quasimorphism $\phi \colon \bG \to \RR$ on $\bG$ is said to be  \emph{$\hG$-quasi-invariant} if there exists a non-negative real number $D'$ such that
\[
\phi(\hg x\hg^{-1}) \sim_{D'} \phi(x)
\]
for every $\hg \in \hG$ and for every $x \in G$. Let $D'(\phi)$ denote the smallest number $D'$ which satisfies the above inequality.
 \item A quasimorphism $\phi \colon \bG \to \RR$ on $\bG$ is said to be \emph{$\hG$-invariant} if $\phi$ is $\hG$-quasi-invariant with $D'(\phi)=0$. In other words,
\[
\phi(\hg x\hg^{-1}) = \phi(x)
\]
holds for every $\hg \in \hG$ and for every $x \in G$.
\end{enumerate}
\end{definition}

\begin{lem} \label{lem=inv}
Let $\hG$ be a group and $\bG$ a normal subgroup of $\hG$.
Assume that $\psi\colon \hG\to \RR$ is a homogeneous quasimorphism on $\hG$. Then the restriction $\psi |_{\bG}\colon \bG\to \RR$ of $\psi$ on $\bG$ is a $\hG$-invariant homogeneous quasimorphism on $\bG$.
\end{lem}

\begin{proof}
 This lemma immediately follows from Lemma~\ref{lem:qm}~(1).
\end{proof}

\begin{remark}\label{remark=homog}
Let $\hG$ be  a group and $\bG$ a normal subgroup of $\hG$. Let $\phi\colon \bG\to \RR$ be a $\hG$-invariant homogeneous quasimorphism on $\bG$. Suppose that $\phi$ is \emph{extendable}, meaning that there exists a quasimorphism $\hat{\phi}$ on $\hG$ such that $\hat{\phi}|_{\bG}=\phi$. Then, we remark that there exists a \emph{homogeneous} extension of $\phi$. To see this, take the homogenization $(\hat{\phi})_{\mathrm{h}}$ of $\hat{\phi}$. Since $\phi\colon \bG\to \RR$ is homogeneous by assumption, the restriction of $(\hat{\phi})_{\mathrm{h}}$ on $\bG$ coincides with that of $\hat{\phi}$. Hence, $(\hat{\phi})_{\mathrm{h}}|_{\bG}=\phi$. Furthermore, by Lemma~\ref{lem=homog}~(3), we have that $D((\hat{\phi})_{\mathrm{h}})\leq 2D(\hat{\phi})$.
\end{remark}

The following lemma was employed in Remark~\ref{rem=fullgenRousseau}.
\begin{lem}[{\cite[Lemma~3.6]{Bavard} (see also \cite[Lemma~2.24]{Ca})}]\label{lem=Bavard}
Let $G$ be a group and $\phi$ a homogeneous quasimorphism on $G$. Then,
\[
D(\phi)= \sup\{|\phi([x,y])|\,|\, x,y\in G\}
\]
holds.
\end{lem}

\subsection{Calabi quasimorphisms}\label{subsec=Calabi}
In this section, we recall the definition of the Calabi property with respect to  quasimorphisms and that of Calabi quasimorphisms in symplectic geometry.
A subset $X$ of a connected symplectic manifold $(M,\omega)$ is said to be \textit{displaceable} if there exists $\varphi \in \Ham(M,\omega)$ satisfying $\varphi(X) \cap \overline{X} = \emptyset$. Here, $\overline{X}$ is the topological closure of $X$  in $M$.

Let $(M,\omega)$ be a $2n$-dimensional  \emph{exact} symplectic manifold, meaning that the symplectic form $\omega$ is exact. For such $(M,\omega)$, we recall that the \textit{Calabi homomorphism}
is a function $\mathrm{Cal}_{M} \colon \Ham(M,\omega)\to\mathbb{R}$ defined by
\[
	\mathrm{Cal}_{M}(\varphi_H)=\int_0^1\int_M H_t\omega^n\,dt,
\]
 where $H \colon [0,1] \times M \to \RR$ is a smooth function. It is known that the Calabi homomorphism is a well-defined group homomorphism (see \cite{Cala,Ban,Ban97,MS,Hum}). The following \emph{Calabi property} plays a key role in this paper.

\begin{definition}[Calabi property with respect to quasimorphisms]\label{definition of Calabi property}
Let $(M,\omega)$ be a $2n$-dimensional  symplectic manifold.
Let $\mu\colon\Ham(M,\omega)\to\mathbb{R}$ be a homogeneous quasimorphism. A non-empty open subset $U$ of $M$ is said to have the \textit{Calabi property} with respect to $\mu$ if $\omega$ is exact on $U$ and if the restriction of $\mu$ to $\Ham(U,\omega)$ coincides with the Calabi homomorphism $\mathrm{Cal}_U$.
\end{definition}
In terms of subadditive invariants, the Calabi property corresponds to the asymptotically vanishing spectrum condition in \cite[Definition 3.5]{KO19}.

\begin{definition}[\cite{EP03,PR}]\label{definition of Calabi qm}
Let $(M,\omega)$ be a $2n$-dimensional  symplectic manifold.
A \textit{Calabi quasimorphism} is defined as a homogeneous quasimorphism $\mu \colon \Ham(M,\omega) \to \RR$ such that every non-empty displaceable open exact subset of $M$ has the Calabi property with respect to $\mu$.
\end{definition}

The first example of Calabi quasimorphism was given by Entov and Polterovich \cite{EP03} by using the Hamiltonian Floer theory.
After their work, Py \cite{Py06} and Brandenbursky \cite{Bra} constructed Calabi quasimorphisms on closed orientable surfaces with higher genus by non-Floer theoretic methods.
We explain the constructions and properties of these Calabi quasimorphisms in Subsections \ref{subsec=mu_P_const} and \ref{subsec=mu_P_property}, respectively.
For other examples of Calabi quasimorphisms; see \cite{Py06t,McD,FOOO,C,LZ,BKS}.

\subsection{Outlined constructions of $\mu_P$ and $\mu_B$}\label{subsec=mu_P_const}

In the present paper, we employ two Calabi quasimorphisms $\mu_P$ and $\mu_B$: the former is due to Py \cite{Py06} and the latter is due to Brandenbursky \cite{Bra}. In this subsection, we outline  their constructions; in the next two subsections, we list the properties of $\mu_P$ and $\mu_B$ needed in the proofs of Theorems~\ref{main thm} and \ref{C0 main thm}.
The proofs only employ the listed properties, and the precise definitions of $\mu_P$ and $\mu_B$ will not appear in the rest of the present paper. Hence, the reader who is mainly interested in the proofs of Theorems~\ref{main thm} and \ref{C0 main thm} can skip this subsection.

We first give an outline of the construction of Py's Calabi quasimorphism \cite{Py06} (see also \cite{R})
when the symplectic form $\omega$ on the surface $S$ is normalized to have the volume $2\nn-2$.

Let $S$ be a closed connected surface whose genus $\nn$ is at least $2$, and $\omega$ a volume form of $S$ such that the volume is $2\nn - 2$.
Choose a metric with constant negative curvature whose area form is $\omega$. Let $\pi \colon P \to S$ be the unit tangent bundle of $S$, $\mathbb{D}$ the Poincar\'e disc, and $S^1 \mathbb{D}$ the unit tangent bundle of $\mathbb{D}$. Set $S^1_\infty = \partial \mathbb{D}$. Define $p_\infty \colon S^1 \mathbb{D} \to S^1_\infty$ to be the map sending $v \in S^1 \mathbb{D}$ to the infinity of the geodesic starting at $v$.
Let $X$ be the vector field generated by the $S^1$-action on $P$. There exists a contact form on $P$ such that $\pi^* \omega = d \alpha$.
Then $X$ is the Reeb vector field of $\alpha$.
Let ${\rm Diff}(P, \alpha)$ denote the group of diffeomorphisms of $P$ preserving $\alpha$, and ${\rm Diff}(P, \alpha)_0$ its identity component.
Now we construct a homomorphism $\Theta \colon \Ham (S, \omega) \to {\rm Diff}_0(P, \alpha)$ as follows. Let $\phi \in \Ham(S,\omega)$ and let $H\colon S\times [0,1] \to\RR$ be a time-dependent Hamiltonian generating $\phi$ and satisfying $\int_S H_t \omega = 0$ for every $t \in [0,1]$. Then set
\[ V_t = \widehat{X}_{H_t} + (H_t \circ \pi)X.\]
Here $\widehat{X}_{H_t}$ denotes the horizontal lift of  the Hamiltonian vector field $X_{H_t}$, {\it i.e.}, $\alpha(\widehat{X}_{H_t}) = 0$ and $\pi_* (\widehat{X}_{H_t}) = X_{H_t}$.
Let  $\Theta(H_t)$ be the isotopy generated by $\{V_t\}_{t\in [0,1]}$, and set $\Theta(\phi) = \Theta(H_1)$.
This is  known to be  well-defined.

Let $\widehat{\Theta(H_t)} \colon S^1 \mathbb{D} \to S^1 \mathbb{D}$ be the lift of $\Theta(H_t)$.
Define $\gamma^{(H,v)} \colon [0,1] \to S^1_\infty$ by $\gamma^{(H,v)}(t) = p_\infty (\widehat{\Theta ( H_t)} (v))$.
Let $\widetilde{\gamma^{(H,v)}}$ be a lift to $\RR$ of $\gamma^{(H,v)}$ and set $\Rot(H,v) = \widetilde{\gamma^{(H,v)}}(1) - \widetilde{\gamma^{(H,v)}}(0)$. Let $\widetilde{\pi}$ denote the projection $S^1 \mathbb{D} \to \mathbb{D}$. For $\widetilde{x} \in \mathbb{D}$, set
\[ \widetilde{\rm angle} (H, \widetilde{x}) = - \inf_{\widetilde{\pi}(v) = \widetilde{x}} \lfloor \Rot(H,v) \rfloor.\]
Since $\widetilde{\rm angle}(H, -)$ is invariant by the action of $\pi_1(S)$, this induces a measurable bounded function ${\rm angle}(H, -)$ on $S$.
Define
\[ \mu_1(\varphi_H) = \int_S {\rm angle}(H, -) \omega.\]
Then $\mu_1$ is a quasimorphism and its homogenization is Py's Calabi quasimorphism $\mu_P$.

Next we review the construction of Brandenbursky's Calabi quasimorphisms.
Let $S$ be a 2-disc $D$ or a closed orientable surface with genus at least two.
Let $B_n(S)$ denote the surface braid group on $S$ on $n$ strands and $P_n(S)$ denote the pure braid subgroup of $B_n(S)$.
Note that $B_n(D)$ is the Artin braid group $B_n$ and $P_n(D)$ is the Artin pure braid group $P_n$.
For a homogeneous quasimorphism $\phi$ on $B_n(S)$, as a generalization of  the Gambaudo--Ghys construction  \cite{GG},
 Brandenbursky  constructed a homogeneous quasimorphism $\mathcal{G}_n(\phi)$ on $\Symp_0(S)$ as follows.

For every pair $x,y \in S$, we choose a geodesic path $s_{xy}$ from $x$ to $y$.
Let $f_t \in \Symp_0(S) $ be an isotopy from $\mathrm{id}$ to $f \in \Symp_0(S)$ and let $z \in S$ be a base point.
For $y \in S$, we define a based loop $\gamma_y$ in $S$ as the concatenation of paths $\gamma_{zy}$, $\{ f_t (y)\}_{0\leq t \leq 1}$, and $\gamma_{f(y)z}$.
Let $X_n(S)$ be the configuration space of ordered $n$-tuples of pairwise distinct points in $S$.
Let $z=(z_1,\dots,z_n) \in X_n(S)$ be a base point.
The fundamental group $\pi_1(X_n(S),z)$ is identified with $P_n(S)$.
For almost  every point $x=(x_1,\dots,x_n) \in X_n(S)$, the $n$-tuple loop $(\gamma_{x_1},\dots,\gamma_{x_n})$ is a based loop in $X_n(S)$.
Since $\Symp_0(S)$ is simply connected, the based homotopy class of this loop does not depend on the choice of the isotopy $\{f_t\}_{0\leq t \leq 1}$.
Let $\gamma(f,x) \in P_n(S)$ be an element represented by this loop.
We define $\Phi \colon \Symp_0(S) \to \RR$ by
\[ \Phi (f) = \int_{x \in X_n(S)} \phi(\gamma(f,x)) dx  \]
and define $\mathcal{G}_n(\phi)$ by the homogenization $\Phi_{\rm h}$ of $\Phi$.

 Let $D$ be the 2-disc with standard symplectic form $\omega_0$.
Let $\phi \colon B_2 \cong \ZZ \to \RR$ be a homomorphism such that $\phi(\sigma_1)=1$, where $\sigma_1 \in B_2$ is the Artin generator.
It is known that the Calabi homomorphism $\mathrm{Cal}_D$ coincides with the map $\mathcal{G}_2(\phi) \colon \Symp_0(D,\omega_0)=\Ham(D,\omega_0) \to \RR$ (see \cite{GG97} for example).

Let $S$ be a closed orientable surface whose genus is at least two, $\omega$ a symplectic form on $S$, and $D \subset S$ an embedded 2-disc. We regard $B_2 = B_2(D)$ as a subgroup of $B_2(S)$.
 Now fix   a homogeneous quasimorphism $\phi$ on $B_2(S)$ such that $\phi(\sigma_1)=1$ for $\sigma_1 \in B_2 \subset B_2(S)$. Here, such a $\phi$ \emph{does} exist; see \cite[Subsection 2.7]{Bra}.
Then $\mathcal{G}_2(\phi) \colon \Ham(S,\omega) \to \RR$ is a Calabi quasimorphism,  and we set $\mu_B = \mathcal{G}_2(\phi)$.

\subsection{Properties of $\mu_P$ and $\mu_B$}\label{subsec=mu_P_property}

In this subsection, we exhibit properties of $\mu_P$ and $\mu_B$ which are used in the proof of Theorems~\ref{main thm} and \ref{C0 main thm}. On continuity  of quasimorphisms, see the next subsection.

\begin{thm}[Known properties of $\mu_P$ and $\mu_B$]\label{thm=mu_P}
Let $S$ be a closed connected orientable surface whose genus $\nn$ is at least two, and $\omega$ a symplectic form on $S$.
Then Py's Calabi quasimorphism $\mu_P\colon \Ham(S,\omega)\to \RR$ and Brandenbursky's Calabi quasimorphism $\mu_B\colon \Ham(S,\omega)\to \RR$ satisfy the following:

\begin{enumerate}[$(1)$]
  \item $\mu_P$ is a $\Symp_0(S,\omega)$-invariant homogeneous quasimorphism on $\Ham(S,\omega)$;
  \item there exists a homogeneous quasimorphism $\hat{\mu}_B$ on $\Symp_0(S,\omega)$ such that $\hat{\mu}_B|_{\Ham(S,\omega)}=\mu_B$. In particular, $\mu_B$ is a $\Symp_0(S,\omega)$-invariant homogeneous quasimorphism on $\Ham(S,\omega)$;
  \item $\mu_P$ and $\mu_B$ are both Calabi quasimorphisms;
  \item for an open subset $U$ of $S$ homeomorphic to the annulus, $U$ has the Calabi property with respect to $\mu_P$.
\end{enumerate}
\end{thm}

For the proof of Theorem~\ref{main thm}, we only employ $\mu_P$; the properties for $\mu_B$ will be utilized in the proof of Theorem~\ref{C0 main thm}.

\begin{proof}
As we explained in Subsection \ref{subsec=mu_P_const}, $\mu_B$ is constructed as the restriction of a homogeneous quasimorphism on  $\Symp_0(S,\omega)$ to $\Ham(S,\omega)$ and thus Item (2) holds.
Here we list the references.
Item (1) is in \cite[Th\'eor\`eme 1]{Py06}.
Item (3) is in \cite[Th\'eor\`eme 1]{Py06} and \cite[Theorem 5]{Bra}.
Item (4) is in \cite[Th\'eor\`eme 2]{Py06}.
\end{proof}

\begin{remark}\label{remark=annulus}
Theorem~\ref{thm=mu_P}~(4) does \emph{not} directly follow from (3). This is because  some annuli are \emph{not} displaceable in $S$. In fact, the proof of Theorem~\ref{thm A} in Section~\ref{section=proof_mainthm}, together with Theorem~\ref{thm=mu_P}~(2), implies that $U$ as in (4) \emph{fails} to have the Calabi property with respect to $\mu_B$.
\end{remark}

\subsection{Continuity of quasimorphisms}\label{subsec=conti}

In this subsection, we discuss the $C^0$-continuity and the $C^0$-discontinuity of quasimorphisms related to $\mu_P$ and $\mu_B$.
The reader who is interested only in the proof of Theorem~\ref{main thm} can skip this subsection.
In this paper, topological groups are not a priori assumed to be Hausdorff.

We employ the following proposition by Entov, Polterovich and Py.

\begin{prop}[{\cite[Proposition~1.4]{EPP}}]\label{prop=conti}
Let $G$ be a topological group and $H$ a dense subgroup of $G$. Let $\mu \colon H \to \RR$ be a homogeneous quasimorphism  on $H$ which is continuous in the relative topology on $H$ from $G$. Then there exists a continuous quasimorphism $\overline{\mu} \colon G \to \RR$ on $G$ such that $\overline{\mu}|_{H} = \mu$.
\end{prop}

We recall a celebrated result by Entov--Polterovich--Py \cite{EPP}, which characterizes the $C^0$-continuity of homogeneous quasimorphisms on $\Ham(S,\omega)$ for a surface $S$.

\begin{thm}[{\cite[Theorem~1.7]{EPP}}]\label{mfh property}
Let $S$ be a closed connected orientable surface and $\omega$ a symplectic form on $S$.
Let $\mu$ be a homogeneous quasimorphism on $\Ham(S,\omega)$.
Then, $\mu$ is continuous in the $C^0$-topology if and only if the following condition is satisfied: there exists a positive number $A$ such that for every disc $D\subset S$ of area less than $A$, the restriction of $\mu$ to $\Ham(D,\omega|_D)$ vanishes.

\end{thm}

\begin{remark}\label{remark=disconti}
Let $S$ be a closed connected orientable surface whose  genus is at least two, and $\omega$ a symplectic form on $S$. Then, $\mu_P$ and $\mu_B$ are \emph{discontinuous} in the $C^0$-topology. Indeed, by Theorem~\ref{thm=mu_P}~(3), a small disc has the Calabi property with respect to $\mu_P$ and $\mu_B$.
By Theorem \ref{mfh property}, this Calabi property serves as an obstruction to the $C^0$-continuity.
\end{remark}

In contrast to Remark~\ref{remark=disconti}, we have the following corollary of  Theorem~\ref{mfh property}, which plays a key role in the proof of Theorem~\ref{C0 main thm}.

\begin{cor}[$C^0$-continuity of $\mu_P-\mu_B$]\label{cor=C^0-conti}
Let $S$ be a closed connected orientable surface whose  genus is at least two, and $\omega$ a symplectic form on $S$. Let $\mu_{P,B}\colon \Ham(S,\omega)\to \RR$ be the quasimorphism on $\Ham(S,\omega)$ defined as $\mu_{P,B}=\mu_P-\mu_B$. Then there exists a homogeneous quasimorphism $\overline{\mu}_{P,B}$ on $\overline{\Ham(S,\omega)}^{C^0}$ that satisfies the following two properties:
\begin{enumerate}[$(1)$]
  \item $\overline{\mu}_{P,B}$ is continuous in the $C^0$-topology;
  \item $\overline{\mu}_{P,B}$ is $\mathrm{Sympeo}_0(S,\omega)$-invariant.
\end{enumerate}
Here, $\mathrm{Sympeo}(S,\omega)$ denotes the $C^0$-closure of $\Symp(S,\omega)$ in the group of homeomorphisms, and
$\mathrm{Sympeo}_0(S,\omega)$ denotes the identity component of $\mathrm{Sympeo}(S,\omega)$. The group $\overline{\Ham(S,\omega)}^{C^0}$ is the $C^0$-closure of $\Ham(S,\omega)$ in $\mathrm{Sympeo}(S,\omega)$.
\end{cor}

\begin{proof}
Let $D$ be a small disc in $S$. Then, by Theorem~\ref{thm=mu_P}~(3), $D$ has the Calabi property with respect to $\mu_P$ and $\mu_B$. Hence, the estimation given by the Calabi property of $D$ cancels for $\mu_{P,B}=\mu_P-\mu_B$. Therefore, $\mu_{P,B}$ fulfills the characterization of the $C^0$-continuity stated in Theorem~\ref{mfh property}.
In other words, $\mu_{P,B}$ is continuous in the relative topology on $\Ham(S,\omega)$ from $\overline{\Ham(S,\omega)}^{C^0}$. By applying Proposition~\ref{prop=conti}, we obtain (1). To prove (2), observe that $\overline{\mu}_{P,B}$ is $C^0$-continuous and that $\mu_{P,B}$ is $\Symp_0(S,\omega)$-invariant. Here, the latter follows from Theorem~\ref{thm=mu_P}~(1) and (2).
\end{proof}

\section{Construction of symplectomorphisms} \label{section=fluxlift}
 In this section, we construct four symplectomorphisms $\sigma_{\qd}$, $\sigma'_{\qd}$, $\tau_{\qd}$, $\tau'_{\qd}$ with compact support on a $1$-punctured torus $P_{\ee}$. Here, $P_{\ee}$ is parameterized by $\ee \in (0,1/4)$, and $\qd=(a,b,c,d)\in \RR^4$ is a quadruple that satisfies
\[
|c|\leq \ee\quad \textrm{and}\quad |d|\leq \ee.
\]
In Section~\ref{section=proof_mainthm}, we embed $P_{\ee}$ into a closed connected oriented surface $S$ in several manners, and obtain symplectomorphisms on $S$. (Therefore, the construction in this section may be regarded as that in a local model.) These elements play a key role to the proofs of Theorems~\ref{thm A} and Theorem~\ref{main thm}; see Section~\ref{section=proof_mainthm}. In Subsection~\ref{subsection=qd}, we explain the construction of $\sigma_{\qd}$, $\sigma'_{\qd}$, $\tau_{\qd}$, $\tau'_{\qd}$. In Subsection~\ref{subsection=properties_qd}, we prove several properties that will be employed in Section~\ref{section=proof_mainthm}.

\subsection{Constructions of $\sigma_{\qd}$, $\sigma'_{\qd}$, $\tau_{\qd}$, $\tau'_{\qd}$ }\label{subsection=qd}
 In this subsection, for $\ee\in \RR$ and for $\qd=(a,b,c,d)\in \RR^4$ such that
\begin{itemize}
  \item $0<\ee<\frac14$,
  \item $|c|\leq \ee$ and $|d|\leq \ee$,
\end{itemize}
we construct four symplectomorphisms $\sigma_{\qd}$, $\sigma'_{\qd}$, $\tau_{\qd}$, $\tau'_{\qd}$ with compact support on a $1$-punctured torus.
 First, we fix $\ee \in (0,1/4)$.  Then  we set $D_{\ee}=([0,1]\times[0,1])\setminus([2\ee, 1-2\ee] \times [2\ee, 1-2\ee])$.
Let $p\colon\RR^2\to\RR^2/\ZZ^2$ be the natural projection
and set $P_{\ee}=p(D_{\ee})$.
We note that $P_{\ee}$ is diffeomorphic to the 1-punctured torus.
We consider $P_{\ee}$ as a connected symplectic manifold with the symplectic form $\omega_0 = dx \wedge dy$, where $(x,y)$ is the standard coordinates on $P_{\ee} \subset \RR^2 / \ZZ^2$.
We define curves $\alpha,\beta\colon[0,1]\to P_{\ee}$ on $P_{\ee}$ by $\alpha(t)=p(0,t)$, $\beta(t)=p(t,0)$.

\begin{lem}\label{def of g}
For every pair of real numbers $(c,d)$ with $|c| \leq \ee$ and $|d| \leq \ee $,
there exist vector fields $Y_c$ and $Y'_d$ on $P_\epsilon$ with compact support such that
\begin{itemize}
  \item $\mathcal{L}_{Y_c}\omega_0=\mathcal{L}_{Y'_d}\omega_0=0$, 
  \item $(Y_c)_{p(x,y)}= c \frac{\partial}{\partial x}$ and $(Y'_d)_{p(x,y)}=  -d  \frac{\partial}{\partial y}$ for every $(x,y)\in ([-\ee ,\ee ]\times\RR)\cup(\RR\times [-\ee,\ee])$.
\end{itemize}
Here, $\mathcal{L}_{Y_c}$ and $\mathcal{L}_{Y'_d}$ are the Lie derivatives with respect to $Y_c$ and $Y'_d$, respectively.
\end{lem}

\begin{figure}[htbp]
  \begin{minipage}[c]{0.45\hsize}
    \centering
    \includegraphics[width=7truecm]{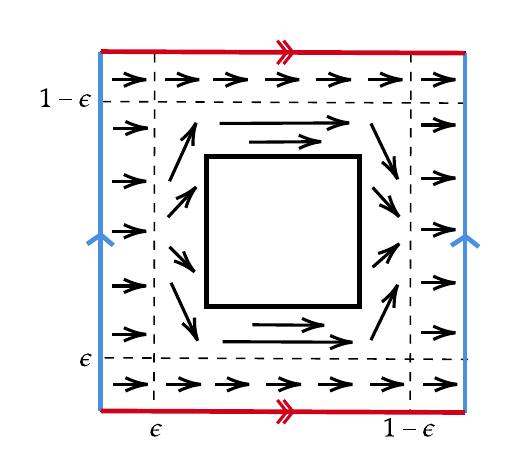}
  \end{minipage}
  \hfill
  \begin{minipage}[c]{0.45\hsize}
    \centering
    \includegraphics[width=7truecm]{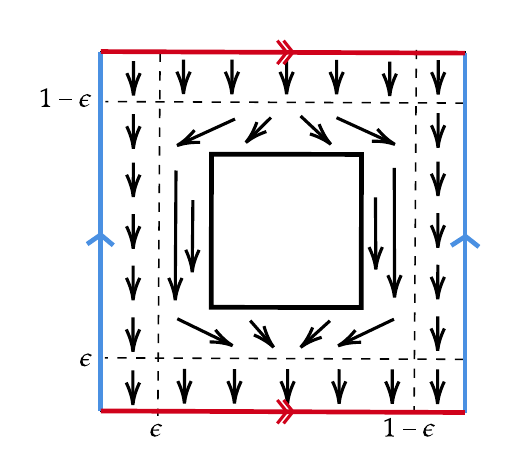}
  \end{minipage}
  \caption{The vector fields $Y_c$ and $Y'_d$  defining $\tau_{\qd}$ and $\tau'_{\qd}$ when $c>0$, $d>0$}
  \label{vec_field_Ya}
\end{figure}

\begin{proof}
Let $G^c$ be a smooth function on $\RR^2$ satisfying the following conditions:
\begin{itemize}
\item
$G^c(x+ m,y+ n)=G^c(x,y)- cn$ for every $(x,y)\in\RR^2$ and $m,n \in \ZZ$.
\item
$G^c(x,y)=-cy$ for every $(x,y)\in ([-\ee ,\ee ]\times\RR)\cup(\RR\times[-\ee ,\ee ])$.
\item
There exists an open neighborhood $U$ of $[2\ee ,1-2\ee  ]\times [2\ee ,1-2\ee ]$ such that
$G^c(x,y)=-  \frac{c}{2}$ for every $(x,y)\in U$.
\end{itemize}
Let $\hat{Y}_c$ be the Hamiltonian vector field generated by $G^c$.
Then, by the first and third conditions on $G^c$, $\hat{Y}_c$ induces the vector field $Y_c$ with compact support on $P_\epsilon$ (Figure \ref{vec_field_Ya}).
Since $\hat{Y}_c$ is a Hamiltonian vector field, we have that $\mathcal{L}_{Y_c}\omega_0=0$.
The second condition on $G^c$ implies that $(Y_c)_{p(x,y)}=c\frac{\partial}{\partial x}$ for every $(x,y)\in ([- \ee, \ee] \times\RR)\cup(\RR\times [-\ee, \ee])$.
We may construct $Y'_d$ in a similar manner.
\end{proof}

For  a pair of real numbers $(q,r)$ with $0 < r \leq |q| \leq \ee$, we define $I_{q,r}$ to be the open interval $(\frac{q-r}{2},\frac{q+r}{2})$ and $J_{q,r}$ to be the closed interval
$[-\frac{|q|-r}{2},\frac{|q|-r}{2}]$.
Note that $I_{-q,r}$, $I_{q,r}$ and $J_{q,r}$ are pairwise disjoint, and their union is a connected interval.

 Now we take a quadruple $\qd=(a,b,c,d) \in \RR^4$ that satisfies $|c|\leq \ee$ and $|d| \leq \ee$. We define  maps $\tau_{\qd}$ and $\tau'_{\qd}$ in the following manner: $\tau_{\qd}$ and $\tau'_{\qd}$ are defined to be the time-1 maps of the flows generated by $Y_c$ and $Y'_d$, respectively.
Since $\mathcal{L}_{Y_c}\omega_0=\mathcal{L}_{Y'_d}\omega_0=0$, we have $\tau_{\qd},\tau'_{\qd} \in \Symp_0(P_\epsilon,\omega_0)$.
We note that $\tau_{\qd}(p(x,y))=p(x+ c,y)$ for every $(x,y)\in (I_{-c,|c|}\times\RR)\cup(\RR \times  [-\ee,\ee] )$ and
$ \tau'_{\qd}(p(x,y))=p(x,y -  d)$ for every $(x,y)\in (\RR\times I_{-d,|d|})\cup( [-\ee,\ee] \times\RR)$.
We note that $\tau_{\qd}={\rm id}$ if $c=0$ and $\tau'_{\qd}={\rm id}$ if $d=0$.

 To define the remaining two symplectomorphisms $\sigma_{\qd}$ and $\sigma'_{\qd}$, we first define smooth functions $H_{q,r}$ and $\hat{H}_{q,r}$. For real numbers $q$ and $r$ with $0< r\leq |q| \leq \ee$, let $\rho_{q,r} \colon [-\ee ,\ee ] \to[-1,1]$ be a smooth function satisfying the following  three conditions:
\begin{enumerate}[(1)]
\item $\Supp(\rho_{q,r})\subset(-\frac{|q|+r}{2},\frac{|q|+r}{2} )$;
\item $\rho_{q,r}(x)+\rho_{q,r}(x+  q)= \frac{q}{|q|}$ for every $x\in I_{-q,r}$;
\item $\rho_{q,r}(x)= \frac{q}{|q|}$ for every $x \in J_{q,r}$.
\end{enumerate}
By the above conditions, we can see that $\rho_{q,r} \equiv  \frac{q}{|q|}$ in a neighborhood of $0$.
For real numbers $q$ and $r$ with $0 <  r\leq |q| \leq \ee$, 
 let $\hat{H}_{q,r}\colon [-\frac12, \frac12] \times \RR \to\RR$ be the function defined by
\[
\hat{H}_{q,r}(x,y)=
\begin{cases}
  -\rho_{q,r}(x) & \text{if $|x| \le \ee$}, \\
  0 & \text{otherwise}.
\end{cases}
\]
Then, $\hat{H}_{q,r}$ induces the smooth function $H_{q,r}\colon P_\epsilon\to\RR$ with compact support. Here, we note that $\hat{H}_{q,r}(\frac12,y)=\hat{H}_{q,r}(-\frac12,y)=0$ for every $y \in \RR$.
For $q=0$, we define $H_{0,r}$ as the function induced by $\hat{H}_{\ee,\ee}$.
Similarly,
let $\hat{H}'_{q,r}\colon \RR \times [-\frac12, \frac12] \to\RR$ be the function defined by
\[
\hat{H}'_{q,r}(x,y)=
\begin{cases}
  -\rho_{q,r}(y) & \text{if $|y| \le \ee$}, \\
  0 & \text{otherwise}.
\end{cases}
\]
Then, $\hat{H}'_{q,r}$ induces the smooth function $H'_{q,r}\colon P_\epsilon\to\RR$ with compact support.
For $q=0$, we define $H'_{0,r}$ as the function induced by $\hat{H}'_{\ee,\ee}$.

 Our definitions of $\sigma_{\qd}$ and $\sigma'_{\qd}$ proceed as follows: we set
\[\Delta= \Delta_{\qd} =
\begin{cases}
  \min\{|c|,|d|\} & \text{if $c\neq 0$ and $d \neq 0$}, \\
  \max\{|c|,|d|\}& \text{otherwise}.
\end{cases}\]
For a real number $t$,
we define $\sigma^t_{\qd} \in\Symp_0(P_\epsilon,\omega_0)$ by
\begin{equation*}\label{fab}
\sigma^t_{\qd}(z)=
\begin{cases}
\varphi_{H_{c,\Delta}}^{bt}(z) & \text{if $z\in p(I_{-c,\Delta}\times\RR)$}, \\
z & \text{otherwise},
\end{cases}
\end{equation*}
and set $\sigma_{\qd}=\sigma^1_{\qd}$.
Similarly, we define $(\sigma'_{\qd})^t \in\Symp_0(P_\epsilon,\omega_0)$ by
\begin{equation*}\label{fab_prime}
(\sigma'_{\qd})^t (z)=
\begin{cases}
 \varphi_{ H'_{d,\Delta}}^{ at }(z) & \text{if $z\in p(\RR \times I_{ -d ,\Delta})$}, \\
z & \text{otherwise},
\end{cases}
\end{equation*}
and set $\sigma'_{\qd}=(\sigma'_{\qd})^1$.
 Then, the vector filed $X_{\qd}$ (resp. $X'_{\qd}$) that generates the flow $\{ \sigma_{\qd}^t \}_{t\in\RR}$ (resp. $\{ (\sigma'_{\qd})^t \}_{t\in\RR}$) is as in Figure \ref{vec_field_X_A}.

\begin{figure}[htbp]
  \begin{minipage}[c]{0.45\hsize}
    \centering
    \includegraphics[width=7truecm]{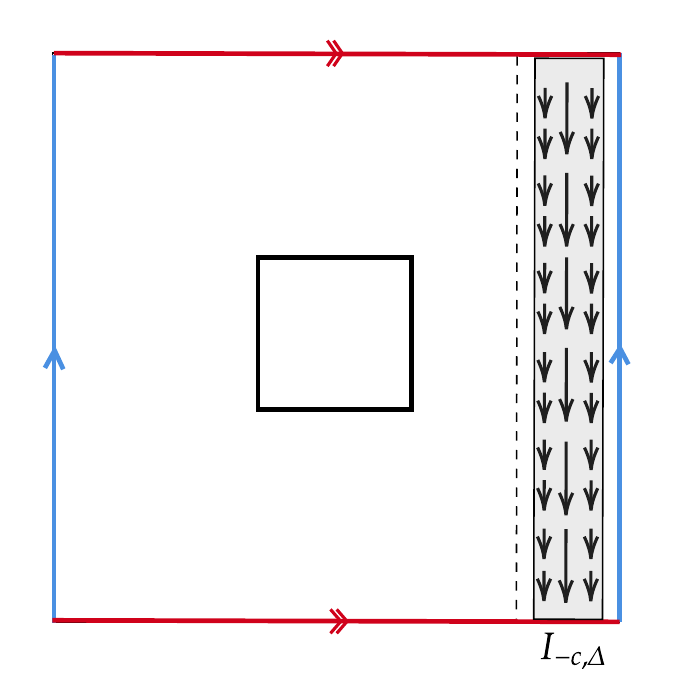}
  \end{minipage}
  \hfill
  \begin{minipage}[c]{0.45\hsize}
    \centering
    \includegraphics[width=7truecm]{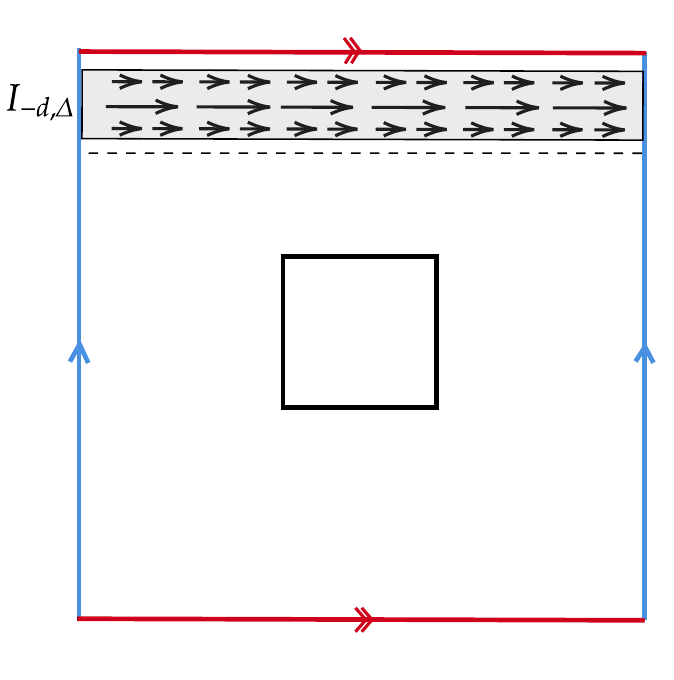}
  \end{minipage}
  \caption{The vector fields $X_{\qd}$ and $X'_{\qd}$  defining $\sigma_{\qd}$ and $\sigma'_{\qd}$ when $a>0$, $b>0$, $c>0$ and $d>0$ }
  \label{vec_field_X_A}
\end{figure}

\subsection{Properties of these symplectomorphisms }\label{subsection=properties_qd}
 In this subsection, we verify several properties of $\sigma_{\qd}$, $\sigma'_{\qd}$, $\tau_{\qd}$, $\tau'_{\qd}$ constructed in Subsection~\ref{subsection=qd}.  Lemma~\ref{lem=sigmatau}, Lemma~\ref{local flux calculation} and Proposition~\ref{prop=flux} are the keys to the proof of Theorem~\ref{thm=non_ext} in Subsection~\ref{subsection=proof_non_ext}. Throughout this subsection, we fix $\ee\in (0,1/4)$.

\begin{lem}\label{lem=sigmatau}
Let $\qd =(a,b,c,d)\in \mathbb{R}^4$ satisfy $|c|\leq \ee$ and $|d|\leq \ee$. Then the following hold:
\begin{enumerate}[$(1)$]
 \item $\sigma_{\qd}$ commutes with $\tau_{\qd}\sigma_{\qd}^{-1}\tau_{\qd}^{-1}$;
 \item $\sigma^{\prime}_{\qd}$ commutes with $\tau^{\prime -1}_{\qd}\sigma^{\prime -1}_{\qd}\tau^{\prime}_{\qd}$;
 \item $\sigma_{\qd}$ commutes with $\tau^{\prime}_{\qd}$, and $\sigma^{\prime}_{\qd}$ commutes with $\tau_{\qd}$.
\end{enumerate}
\end{lem}

\begin{proof}

  First, we prove (1). If $c = 0$, then $\tau_{\qd} = {\rm id}$ and thus the claim is clear.
  If $c\neq0$, then $\Supp(\sigma_{\qd}) \subset p( I_{-c,\Delta} \times \RR )$ and
  $\Supp(\tau_{\qd}\sigma_{\qd}^{-1}\tau_{\qd}^{-1}) \subset p( I_{c,\Delta} \times \RR )$ are disjoint. Hence we obtain the claim.
  We can show (2) as well as (1).
  Finally, we prove (3). Since $\Supp(\sigma_{\qd}) \subset p([-\ee,\ee] \times \RR )$, the values of $H_{c,\Delta}$ depend only on the $x$-coordinate and $\tau'|_{p([-\ee,\ee] \times \RR )}$ is generated by the vector field $-d \frac{\partial}{\partial y}$ on  $p([-\ee,\ee] \times \RR )$,
   $\sigma_{\qd}$ commutes with $\tau^{\prime}_{\qd}$. Similarly, we can show that $\sigma^{\prime}_{\qd}$ commutes with $\tau_{\qd}$. \qedhere

\end{proof}

 Since the center of the fundamental group of $P_\epsilon$ is trivial, \cite[Proposition 5.1]{Fa} implies that the flux group of $(P_\epsilon,\omega_0)$ is zero and thus the flux homomorphism gives a group homomorphism $\flux_{\omega_0}$ from $\Symp_0(P_\epsilon,\omega_0)$ to $H_c^1(P_\epsilon;\RR)$. 
The values of the flux homomorphism at $\sigma_{\qd}$, $\sigma'_{\qd}$, $\tau_{\qd}$ and $\tau'_{\qd}$ are determined as follows.

\begin{lem}\label{local flux calculation}
Let $\qd =(a,b,c,d)\in \mathbb{R}^4$ satisfy $|c|\leq \ee$ and $|d|\leq \ee$. Then,  we have that
\begin{gather*}
  { \flux }_{\omega_0}\left( \sigma_{\qd} \right)=b[\beta]^\ast, \quad
  { \flux }_{\omega_0}\left( \sigma'_{\qd} \right)=a[\alpha]^\ast,\\
  { \flux }_{\omega_0}\left( \tau_{\qd} \right)=c[\alpha]^\ast, \quad
  { \flux }_{\omega_0}\left( \tau'_{\qd} \right)=d[\beta]^\ast.
\end{gather*}
 Here, $([\alpha]^*,[\beta]^*)$ denotes the dual basis to the basis $([\alpha],[\beta])$ of  $H_1(P_\ee;\RR).$
\end{lem}

\begin{proof}
First, we assume that $c\neq0$.
 We recall that $X_{\qd}$ is the vector field generated by the flow $\{ \sigma_{\qd}^t \}_{t\in\RR}$ (Figure \ref{vec_field_X_A}).
Then, we have that
\begin{equation*}
\left(X_{\qd}\right)_z=
\begin{cases}
b\cdot\left(X_{H_{c,\Delta}}\right)_z & \text{if }z\in p(I_{-c,\Delta}\times\RR), \\
0 & \text{otherwise}.
\end{cases}
\end{equation*} Hence,
\begin{equation*}
\left(\iota_{X_{\qd}}\omega_0\right)_{z=p(x,y)}=
\begin{cases}
  \displaystyle
-\left(d(b\cdot H_{c,\Delta})\right)_z=b \frac{d\rho_{c,\Delta}}{dx} (x)dx & \text{if }z\in p(I_{-c,\Delta}\times\RR), \\
0 & \text{otherwise}.
\end{cases}
\end{equation*}
Since $\rho_{c,\Delta}$ is constant in a neighborhood of $0$,
we have that $ \frac{d\rho_{c,\Delta}}{dx} (0)=0$; hence
\[\left[\left(\iota_{X_{\qd}}\omega_0\right)\right]([\alpha])=\int_\alpha\iota_{X_{\qd}}\omega_0=\int_ \alpha b \frac{d\rho_{c,\Delta}}{dx} (0)  dx = 0.\]
Since
\[\rho_{c,\Delta} \left( \frac{-c-\Delta}{2} \right) =
\begin{cases}
0 & \text{if $c > 0$}, \\
-1 & \text{if $c < 0$},
\end{cases}
\quad \text{and} \quad
\rho_{c,\Delta} \left( \frac{-c+\Delta}{2} \right) =
\begin{cases}
1 & \text{if $c > 0$}, \\
0 & \text{if $c < 0$},
\end{cases}
\]
we have that
\begin{align*}
      &\left[\left(\iota_{X_{\qd}}\omega_0\right)\right]([\beta]) \\
      &=\int_\beta\iota_{X_{\qd}}\omega_0\\
      & =\int_{I_{-c,\Delta}}b \frac{d\rho_{c,\Delta}}{dx} (x)dx \\
      & = b\left(\rho_{c,\Delta} \left(\frac{-c+\Delta}{2} \right)-\rho_{c,\Delta} \left(\frac{-c-\Delta}{2} \right)\right) \\
      &=b.
\end{align*}
Hence we obtain that
\[{ \flux }_{\omega_0}\left(\sigma_{\qd} \right)=b[\beta]^\ast.\]
If $c=0$, then $H_{c,\Delta}(z) = - \rho_{\ee,\ee}(x) $ for $z=p(x,y) \in p(I_{-\ee,\ee}\times \RR)$.
By the same argument as above, we obtain ${ \flux }_{\omega_0}\left(\sigma_{\qd} \right)=b[\beta]^\ast.$
 Similarly, we may obtain that ${ \flux }_{\omega_0}\left(\sigma'_{\qd}\right)=a[\alpha]^{\ast}$.

 Next, we proceed  to calculate ${ \flux }_{\omega_0}\left(\tau_{\qd}\right)$.
By the definition of $Y_c$,
\begin{equation*}
\left(\iota_{Y_c}\omega_0\right)_{z=p(x,y)}= c dy
\end{equation*}
for every $(x,y)\in ([-\ee,\ee]\times\RR)\cup(\RR \times  [-\ee,\ee] )$.
Therefore,  since $\alpha^\ast dy=dt$ and $\beta^\ast dy=0$,
\begin{align*}
      \left[\left(\iota_{Y_c}\omega_0\right)\right]([\alpha])&=\int_\alpha\iota_{Y_c}\omega_0=\int_ \alpha  c dy=c,\\
       \left[\left(\iota_{Y_c}\omega_0\right)\right]([\beta])&=\int_\beta\iota_{Y_c}\omega_0=\int_\beta c dy=0.
\end{align*}
Hence,  we obtain that
\[{ \flux }_{\omega_0}\left( \tau_{\qd} \right)=c[\alpha]^\ast,\]
as desired.
 We may have that ${ \flux }_{\omega_0}\left(\tau'_{\qd}\right)=d[\beta]^{\ast}$ in a  manner similar  to one above.
\end{proof}

 The next goal is to prove Proposition~\ref{prop=flux}.
\begin{lem}\label{ham representation}
  Let $\qd =(a,b,c,d)\in \mathbb{R}^4$ satisfy $|c|\leq \ee$ and $|d|\leq \ee$.
  \begin{enumerate}[$(1)$]
    \item   If $c \neq 0$, then
      $[\sigma_{\qd},\tau_{\qd}]=\varphi_{H_{c,\Delta}}^b $.
    \item   If $d \neq 0$, then
    $[\sigma'_{\qd},{\tau}_{\qd}^{\prime -1}]=\varphi_{ H'_{d,\Delta} }^{a}.$
  \end{enumerate}
\end{lem}

\begin{proof}
Since $\rho_{c,\Delta}(x)+\rho_{c,\Delta}(c+x)=\frac{c}{|c|}$ for every $x\in I_{-c,\Delta}$,  we deduce from the definition of $\sigma_{\qd}$  that
\begin{equation*}
\tau_{\qd} \sigma_{\qd}^{-1} \tau_{\qd}^{-1}(z)=
\begin{cases}
\varphi_{H_{c,\Delta}}^b(z) & \text{if }z\in p\left(I_{c,\Delta}\times\RR\right), \\
z & \text{otherwise}.
\end{cases}
\end{equation*}

\begin{figure}[htbp]
  \begin{minipage}[c]{0.45\hsize}
    \centering
    \includegraphics[width=7truecm]{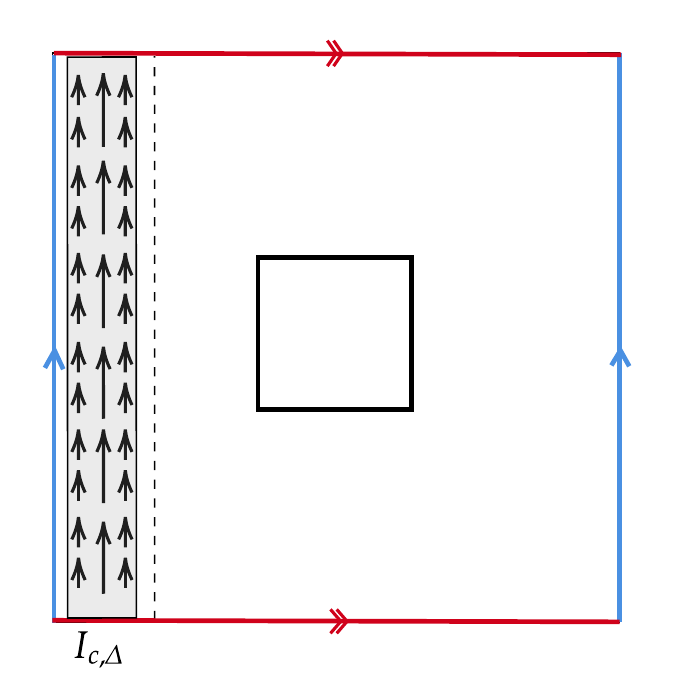}
  \end{minipage}
  \hfill
  \begin{minipage}[c]{0.45\hsize}
    \centering
    \includegraphics[width=7truecm]{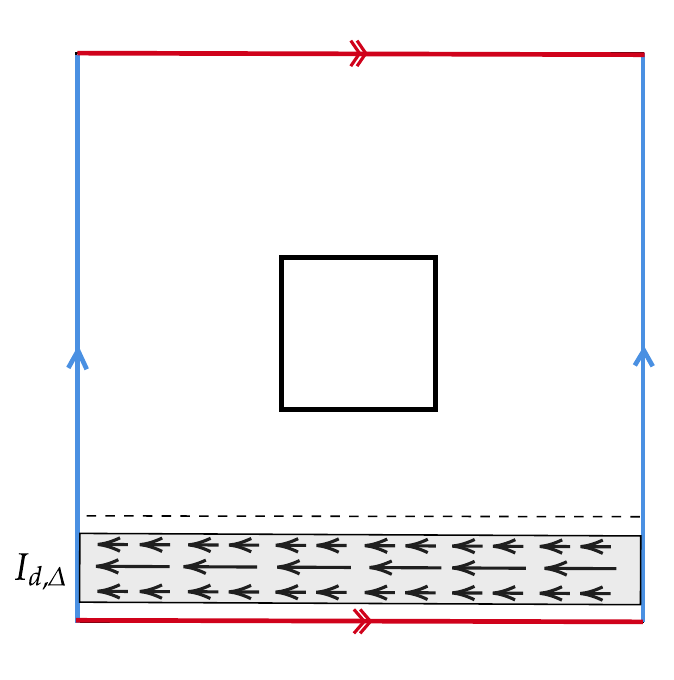}
    \end{minipage}
  \caption{The vector fields generating the flows $\{\tau_{\qd}\sigma_{\qd}^{-t}\tau_{\qd}^{-1}\}_{t\in\RR}$ and $\{\tau_{\qd}^{\prime -1}(\sigma_{\qd}^{\prime})^{-t}\tau_{\qd}^{\prime}\}_{t\in\RR}$ }
  \label{flow}
\end{figure}

 Hence on $[\sigma_{\qd},\tau_{\qd}]=\sigma_{\qd} \left( \tau_{\qd}\sigma_{\qd}^{-1}\tau_{\qd}^{-1} \right)$, we have that
\begin{equation*}
[\sigma_{\qd},\tau_{\qd}](z)=
\begin{cases}
\varphi_{H_{c,\Delta}}^b(z) & \text{if }z\in p\left((I_{c,\Delta}\sqcup I_{-c,\Delta})\times\RR\right), \\
z & \text{otherwise}.
\end{cases}
\end{equation*}
By observing that $\varphi_{H_{c,\Delta}}^b=\mathrm{id}$ on $P_{\ee}\setminus p\left((I_{c,\Delta}\sqcup I_{-c,\Delta})\times\RR\right)$, we conclude that $[\sigma_{\qd},\tau_{\qd}]=\varphi_{H_{c,\Delta}}^b$.
This completes the proof of (1). The proof of (2) is similar to one above.

\end{proof}

\begin{lem}\label{local ham integration}
 For every pair of real numbers $(q,r)$ with $0<r\leq |q| \leq \ee$,
\[\int_{P_\epsilon}H_{q,r}\omega_0=-q  \quad \textrm{and}\quad \int_{P_\epsilon}H'_{q,r}\omega_0=-q.\]
\end{lem}

\begin{proof}

By $\Supp(\rho_{q,r})\subset I_{-q,r} \sqcup J_{q,r} \sqcup I_{q,r}$,
\[\int_{P_\epsilon}H_{q,r}\omega_0=\int_{p((I_{-q,r}  \sqcup  J_{q,r}  \sqcup  I_{q,r})\times\RR)}H_{q,r}\omega_0
= - \int_{I_{-q,r}  \sqcup  J_{q,r}  \sqcup  I_{q,r}} \rho_{q,r}(x) dx .\]
Since $\rho_{q,r}(x)+\rho_{q,r}(q+x)=\frac{q}{|q|}$ for every $x\in I_{-q,r}$, we have that
\begin{align*}
      - \int_{I_{-q,r}  \sqcup  J_{q,r}  \sqcup  I_{q,r}} \rho_{q,r}(x) dx  &= - \int_{I_{-q,r}  \sqcup  I_{q,r}} \rho_{q,r}(x) dx - \int_{J_{q,r}} \rho_{q,r}(x) dx \\
      & = - r \cdot \frac{q}{|q|} - (|q|-r) \cdot \frac{q}{|q|} \\
      &=-q,
\end{align*}
 as desired. The same argument as one above with switching the two coordinates verifies that  $\int_{P_\epsilon}H'_{q,r}\omega_0=-q$.
\end{proof}

 The next proposition treats the values of certain commutators associated with $\sigma_{\qd}$, $\tau_{\qd}$, $\sigma'_{\qd}$, $\tau'_{\qd}$ by Py's Calabi quasimorphisms $\mu_P$. Recall the construction and properties of $\mu_P$ from Subsections~\ref{subsec=mu_P_const} and \ref{subsec=mu_P_property}.

\begin{prop} \label{prop=flux}
 Let $S$ be a closed connected orientable surface whose genus is at least two and $\omega$ a symplectic form on $S$. Assume that  $\iota\colon (P_{\ee},\omega_0)\to (S,\omega)$ is a symplectic embedding. Let $\mu_P\colon \Ham(S,\omega)\to \RR$ be the Calabi quasimorphism of Py.
Let $\qd =(a,b,c,d)\in \mathbb{R}^4$ satisfy $|c|\leq \ee$ and $|d|\leq \ee$.
Then the following hold:
\begin{enumerate}[$(1)$]
  \item $ \mu_P \left(\iota_{\ast}\left([\sigma_{\qd},\tau_{\qd}]\right)\right)=-bc$;
  \item $ \mu_P \left(\iota_{\ast}\left([\sigma^{\prime}_{\qd},\tau^{\prime -1}_{\qd}]\right)\right)=-ad$.
\end{enumerate}
\end{prop}

 For the proof of Proposition~\ref{prop=flux}, Theorem \ref{thm=mu_P}~(4) plays the key role: it asserts that an open subset $U$ of $S$ homeomorphic to the annulus has the Calabi property with respect to $\mu_P$.

\begin{proof}[Proof of Proposition~$\ref{prop=flux}$]
  First, if $c=0$, we have that $[\sigma_{\qd},\tau_{\qd}]=\mathrm{id}$.
  Hence $ \mu_P  \left(\iota_{\ast}\left([\sigma_{\qd},\tau_{\qd}]\right)\right)=-bc=0$.
   In what follows,  we may assume that $c\neq0$.
 Observe that  there exists an open subset $U$ of $P_{\ee}$ such that $\Supp(H_{c,\Delta})\subset U$ and that $U$ is homeomorphic to the annulus.  By Theorem \ref{thm=mu_P}~(4), $\iota(U)$ has the Calabi property with respect to $\mu_P$. Therefore, Lemma \ref{ham representation} implies that
  \[  \mu_P \left(\iota_\ast\left([\sigma_{\qd},\tau_{\qd}]\right)\right)= \mu_P \left(\iota_\ast\left(\varphi_{H_{c,\Delta}}^b\right)\right)=\int_{\iota(P_\epsilon)}\iota_\ast\left(bH_{c,\Delta}\right)\omega=\int_{P_\epsilon}bH_{c,\Delta}\omega_0.
  \]
  Hence, by Lemma \ref{local ham integration}, we have that
  \[ \mu_P \left(\iota_\ast\left([\sigma_{\qd},\tau_{\qd}]\right)\right) = b\int_{P_\epsilon}H_{c,\Delta}\omega_0 = -bc. \]
  Hence (1) follows. We may obtain (2) in a similar way.
\end{proof}

\section{Proofs of Theorem \ref{thm A} and the Main Theorem}\label{section=proof_mainthm}

Recall that $S$ is a closed orientable surface whose genus $l$ is at least two. We take curves $\alpha_1,\beta_1,\alpha_2, \beta_2, \ldots,\alpha_\nn,\beta_\nn \colon [0,1] \to S$ on $S$ as depicted in Figure~\ref{curves}.
Let $[\alpha_1]^\ast,[\beta_1]^\ast,[\alpha_2]^\ast, [\beta_2]^\ast, \ldots,[\alpha_\nn]^\ast, [\beta_\nn]^\ast \in H^1(S ; \RR)$ be the dual basis of $[\alpha_1], [\beta_1],[\alpha_2], [\beta_2], \ldots, [\alpha_\nn], [\beta_\nn] \in H_1(S ; \RR)$.
\begin{figure}[htbp]
\centering
\includegraphics[width=7truecm]{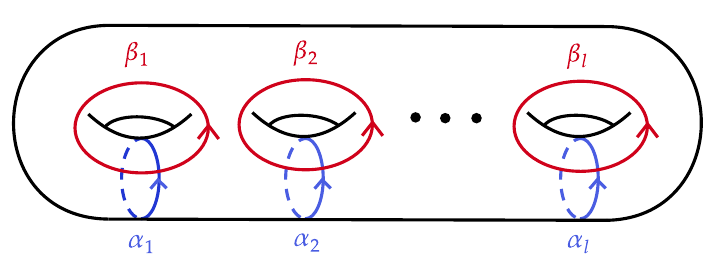}
\caption{$\alpha_1,\beta_1, \alpha_2, \beta_2, \ldots,\alpha_\nn,\beta_\nn\colon[0,1]\to S$}
\label{curves}
\end{figure}

 Recall that $ \bb \colon H^1(S;\RR)\times H^1(S;\RR)\to \RR$ denotes the intersection form:
let $v,w\in H^1(S;\RR)$. Express $v$ and $w$ as $v=\sum_{j=1}^{\nn}(a_j[\alpha_j]^\ast+b_j[\beta_j]^\ast)$  ($a_j,b_j\in \RR$)
and $w=\sum_{j=1}^{\nn}(c_j[\alpha_j]^\ast+d_j[\beta_j]^\ast)$  ($c_j,d_j\in \RR$), respectively. Then, the intersection number $\bb(v,w)$ equals
\[
 \bb(v,w)=\sum_{j=1}^{\nn}(a_jd_j-b_jc_j).
\]


Here we restate Theorem~\ref{main thm} for the convenience of the reader:
\begin{thm}[Restatement of Theorem~$\ref{main thm}$]\label{thm=main}
 Let $S$ be a closed orientable surface whose genus $\nn$ is at least two and  $\omega$ a symplectic form on $S$. Assume that $f,g \in \Symp_0(S,\omega)$ satisfy $fg = gf$. Then the cup product $\flux_\omega(f) \smile \flux_\omega(g)$ equals $0$.
\end{thm}
Recall from Section~\ref{intro section} the strategy of the proof of Theorem~\ref{main thm}. We will demonstrate Theorem~\ref{main thm} by combining two results on the extendability of Py's Calabi quasimorphism $\mu_P$; one in the \emph{affirmative} and the other in the \emph{negative}.
The former may be seen as an extension theorem for discrete groups, which is proved by the authors in \cite[Proposition 1.6]{KKMM20}. Recall from Subsection~\ref{subsec=qm} (Definition~\ref{defn=inv_qm}) the relevant definitions.

\begin{prop}[Extension theorem for discrete groups]\label{prop=discrete}
Assume that the short exact sequence $1\to \bG \to \hG \xrightarrow{q} Q \to 1$ of groups \emph{virtually splits}. Namely, there exists a subgroup $Q_0$ of $Q$ of finite index and a group homomorphism $s_0\colon Q_0\to \hG$ such that $q\circ s_0=\mathrm{id}_{Q_0}$. Then, for every $\hG$-quasi-invariant quasimorphism $\phi$ on $\bG$, there exists a quasimorphism $\hat{\phi}$ on $\hG$ such that $\hat{\phi}|_\bG = \phi$ and $D(\hat{\phi}) \le D(\phi) + 3 D'(\phi)$.
\end{prop}

\begin{remark}\label{remark=ext_homog}
By Remark~\ref{remark=homog}, which employs the homogenization (Lemma~\ref{lem=homog}), Proposition~\ref{prop=discrete} implies the following: assume that $(\hG,\bG,Q=\hG/\bG)$ satisfies the assumption of Proposition~\ref{prop=discrete}.
If $\phi$ is a $\hG$-invariant homogeneous quasimorphism, then there exists a \emph{homogeneous} quasimorphism $\psi$ on $\hG$ such that $\psi|_\bG = \phi$ and $D(\psi) \le 2D(\phi)$.  Here, we note that $D'(\phi)=0$ by  Lemma~\ref{lem:qm}~(1).
\end{remark}

The latter is a precise version of Theorem~\ref{thm A}, which proves the \emph{non-}extendability of $\mu_P$. Let $S$ be a closed orientable surface whose genus $\nn$ is at least two. Take $2\nn$ curves $\alpha_1,\beta_1,\ldots ,\alpha_{\nn},\beta_{\nn}$ as in Figure~\ref{curves}.
Let $[\alpha_1]^\ast,[\beta_1]^\ast,\ldots,[\alpha_\nn]^\ast,[\beta_\nn]^\ast \in H^1(S ; \RR)$ be the dual basis of $[\alpha_1], [\beta_1],\ldots, [\alpha_\nn],[\beta_\nn] \in H_1(S ; \RR)$.
In the present paper, for $w\in H^1(S;\RR)$, we use the symbol $|w|_{\max}$ as
\[
|w|_{\max}:=\max\{|c_j|,|d_j|\, \mid\, j = 1, \dots, \nn \},
\]
where $w=\sum_{j=1}^{\nn}(c_j[\alpha_j]^{\ast}+d_j[\beta_j]^{\ast})$.

\begin{thm}[Precise version of Theorem~\ref{thm A} with explicit $k_0$]\label{thm=non_ext}
Let $S$ be a closed connected orientable surface whose genus $\nn$ is at least two and $\omega$ a symplectic form on $S$.
Let $\bar{v}$, $\bar{w} \in H^1(S;\RR)$ with $\bar{v} \smile \bar{w}  \neq 0$. For a positive integer $k$, set $\Lambda_k = \langle \bar{v}, \bar{w} / k\rangle$ and $G_{\Lambda_k} = \flux_\omega^{-1}(\Lambda_k)$.
Let $k_0= 8 \nn \max\{ 1, \left\lceil\frac{|\bar{w}|_{\max}}{{\rm area}(S)} \right\rceil  \} $,
 where $\lceil \cdot\rceil$  is the ceiling function.

Then, for every $k \ge k_0$, Py's Calabi quasimorphism $\mu_P$ is \emph{not} extendable to $G_{\Lambda_k}$; namely, for every $k\geq k_0$, there does \emph{not} exist a homogeneous quasimorphism $\mu_P^{\Lambda_k}$ on $G_{\Lambda_k}$ such that $\mu_P^{\Lambda_k}|_{\Ham(S,\omega)}=\mu_P$.
\end{thm}
 For the case where $\mathrm{area}(S)> 1$, we in fact obtain a better bound on $k_0$; see Remark~\ref{remark=normalize}.

 In Subsection~\ref{subsection=proof_main}, we
deduce Theorem~\ref{main thm} from Proposition~\ref{prop=discrete} and Theorem~\ref{thm=non_ext}; it is almost straightforward. Since Proposition~\ref{prop=discrete} is already proved in \cite{KKMM20}, Theorem~\ref{thm=non_ext} may be seen as the essential part of Theorem~\ref{main thm}.  Subsection~\ref{subsection=proof_non_ext}  is devoted to the proof of Theorem~\ref{thm=non_ext} (precise version of Theorem~\ref{thm A}).

\subsection{Proof of Theorem~\ref{thm A} }\label{subsection=proof_non_ext}
 In this subsection, we prove Theorem~\ref{thm=non_ext} (precise version of Theorem~\ref{thm A}). Throughout this subsection, we assume the setting of Theorem~\ref{thm=non_ext}.
Set $\epsilon =  \frac1{8\nn} \min\{ 1 , {\rm area}(S) \} $.
Then, $0< \ee<\frac14 $ and ${\rm area}(P_{\ee}) \cdot \nn < 8 \ee \nn \leq {\rm area}(S)$.
Hence, there exist symplectic embeddings $\iota^1,\ldots,\iota^\nn\colon$ $(P_\epsilon,\omega_0)\to (S,\omega)$ satisfying the following conditions:
\begin{enumerate}[(1)]
  \item $\iota^{j} \circ \alpha = \alpha_{j}$, $\iota^{ j } \circ \beta = \beta_{ j }$ for $j = 1,\ldots,\nn$.
  \item $\iota^i(P_\epsilon)\cap\iota^j(P_\epsilon)=\emptyset$ if $i\neq j$.
\end{enumerate}
We set $v=\bar{v}$ and $w=\bar{w}/k$ and represent them as
\begin{equation}\label{eq:vw}
  v = \sum_{j=1}^{\nn}(a_j[\alpha_j]^{\ast}+b_j[\beta_j]^{\ast})
 \quad \text{and} \quad w = \sum_{j=1}^{\nn}(c_j[\alpha_j]^{\ast}+d_j[\beta_j]^{\ast}).
\end{equation}
Set $\qd_j = (a_j,b_j,c_j,d_j)$ for $j =1 ,\dots, \nn$.
Then, the assumption $k \geq k_0$ implies that $|c_j|\leq \ee$ and $|d_j|\leq \ee$ for $j=1,\dots,\nn$ since $ |w|_{\max}/k\leq |w|_{\max}/k_0 \leq \ee$.
Hence, we can construct $\sigma_{\qd_j}$, $\tau_{\qd_j}$, $\sigma'_{\qd_j}$, and $\tau'_{\qd_j}$ defined in Subsection \ref{subsection=qd}.

To prove Theorem~\ref{thm=non_ext}, we  first show the following Lemma~\ref{lemma=haitteru}, Proposition~\ref{prop=ext_estimate} and Proposition~\ref{prop=intersection}.
For $m \in \NN$, we define $\fk_m, \gk, \hk \in \Symp_0(S,\omega)$ by
\begin{equation}\label{eq:fg}
  \fk_m = \prod_{j=1}^l \iota^j_{\ast}(\sigma_{\qd_j}^m \sigma_{\qd_j}^{\prime m}), \quad
 \gk = \prod_{j=1}^l \iota^j_{\ast}(\tau_{\qd_j}),  \quad \text{and} \quad
 \hk = \prod_{j=1}^l \iota^j_{\ast}(\tau'_{\qd_j});
\end{equation}
 recall the symbol $\prod$ from the end of Section~\ref{intro section}.

\begin{lem}\label{lemma=haitteru}
 Let $\fk_m$, $\gk$, and $\hk$ be as in \eqref{eq:fg} and $G_{\Lambda_k}$ as in the setting of Theorem $\ref{thm=non_ext}$.
Then the following hold.
\begin{enumerate}[$(1)$]
 \item For every $m\in \NN$, $\fk_m \in G_{\Lambda_k}$.
 \item $\gk \hk \in G_{\Lambda_k}$.
\end{enumerate}
\end{lem}
\begin{proof}

We note that for every $\phi \in \Symp_0(P_\epsilon,\omega)$ and every symplectic embedding $\iota \colon (P_\epsilon,\omega_0) \to (S,\omega)$,
\[ \flux_\omega(\iota_\ast\phi) =
 \iota_{\ast} \left( \flux_{\omega_0}(\phi) \right) ;  \]
 recall from Subsection~\ref{subsec=symp} that $\Gamma_\omega=\{0\}$.

For  every $m\in \NN$,
\begin{align*}
\flux_{\omega}(\fk_m)&=\sum_{j= 1 }^l\flux_{\omega}(\iota^j_{\ast}(\sigma_{\qd_j}^m \sigma_{\qd_j}^{\prime m}))\\
&=m\left\{\sum_{j= 1 }^l\flux_{\omega}(\iota^j_{\ast}(\sigma_{\qd_j}))+\sum_{j= 1 }^l\flux_{\omega}(\iota^j_{\ast}(\sigma'_{\qd_j}))\right\}.
\end{align*}
By Lemma~\ref{local flux calculation}, we have that
\[
\flux_{\omega}(\iota^j_{\ast}(\sigma_{\qd_j}))=b_j[\beta_j]^{\ast},\quad \flux_{\omega}(\iota^j_{\ast}(\sigma'_{\qd_j}))=a_j[\alpha_j]^{\ast}
\]
for every $1\leq j\leq l$. Here recall that $\iota^{j}\circ \alpha=\alpha_{j}$ and $\iota^{j}\circ \beta=\beta_{j}$. Hence, we conclude that
\[
\flux_{\omega}(\fk_m)=mv \in \Lambda_k,
\]
which proves (1). In a  manner similar  to one above, we have that
\[
\flux_{\omega}(\gk \hk) = w \in \Lambda_k;
\]
hence we show (2).
\end{proof}

For $m \in \NN$, we set $\gamma_m \in \Symp_0(S,\omega)$ by
\begin{equation} \label{eq:gamma}
  \gamma_m = [\fk_m,\gk][\fk_m,\hk^{-1}]^{-1}.
\end{equation}

\begin{prop}\label{prop=ext_estimate}
 Let $\gamma_m$ be as in \eqref{eq:gamma} and $G_{\Lambda_k}$ as in the setting of Theorem $\ref{thm=non_ext}$.
Let $\hat{\phi}\colon G_{\Lambda_k}\to \RR$ be a homogeneous quasimorphism on $G_{\Lambda_k}$. Then for every $m\in \NN$,
\[
\hat{\phi}(\gamma_m)\sim_{3D(\hat{\phi})}0
\]
holds true.
\end{prop}

\begin{proof}
  Since $\hat\phi$ is a quasimorphism on $G_{\Lambda_k}$, Lemma~\ref{lemma=haitteru}~(1) implies that
  \begin{gather*}
    \hat\phi([\fk_m,\gk])  = \hat\phi( \fk_m \gk \fk_m^{-1} \gk^{-1})  \sim_{D(\hat\phi)} \hat\phi( \fk_m  ) + \hat\phi( \gk \fk_m^{-1}\gk^{-1} ), \\
    \hat\phi([\fk_m,\hk^{-1}])  = \hat\phi( \fk_m \hk^{-1} \fk_m^{-1} \hk) \sim_{D(\hat\phi)} \hat\phi( \fk_m ) + \hat\phi( \hk^{-1} \fk_m^{-1}\hk ).
  \end{gather*}
     Recall from Lemma~\ref{lemma=haitteru}~(2) that $\gk \hk \in G_{\Lambda_k}$. Since $\hat\phi$ is homogeneous, by Lemma \ref{lem:qm}~(1) we have that
\begin{align*}
\hat\phi( \gk \fk_m^{-1}\gk^{-1} ) &= \hat\phi(  (\gk\hk) \hk^{-1} \fk_m^{-1}\hk (\gk\hk)^{-1}  ) \\
&= \hat\phi(  \hk^{-1} \fk_m^{-1}\hk  ).
\end{align*}
  Therefore,
  \begin{align*}
    \hat{\phi}(\gamma_m) & \sim_{D(\hat\phi)}
      \hat\phi([\fk_m,\gk]) + \phi([\fk_{m},\hk^{-1}]^{-1}) \\
      & = \hat\phi([\fk_m,\gk]) - \phi([\fk_{m},\hk^{-1}]) \\
      & \sim_{2D(\hat\phi)} \hat\phi( \gk \fk_m^{-1}\gk^{-1} ) - \hat\phi( \hk^{-1} \fk_m^{-1}\hk ) =  0. \qedhere
  \end{align*}
\end{proof}

\begin{prop}\label{prop=intersection}
 Let $\fk_m$, $\gk$, and $\hk$ be as in \eqref{eq:fg}, $\gamma_m$ as in \eqref{eq:gamma}, and $G_{\Lambda_k}$ as in the setting of Theorem $\ref{thm=non_ext}$.
Then the following hold true for every $m\in \NN$:
\begin{enumerate}[$(1)$]
  \item $\mu_P([{\fk}_m,\gk])=-m\sum_{j=1}^{\nn} b_jc_j$;
  \item $\mu_P([{\fk}_m,{\hk}^{-1}])=-m\sum_{j=1}^{\nn}  a_jd_j$;
  \item $\mu_P(\gamma_m)\sim_{D(\mu_P)}m \bb (v,w)$.
\end{enumerate}
\end{prop}

\begin{proof}
 By Lemma \ref{lem=sigmatau}~(3), $\sigma'_{\qd_j}$ commutes with $\tau_{\qd_j}$. By Lemma \ref{lem=sigmatau}~(1), $\sigma_{\qd_j}$ commutes with $\tau_{\qd_j} \sigma_{\qd_j}^{-1}\tau_{\qd}^{-1}$. Therefore,
    \begin{align*}
      [\fk_m,\gk] & = \prod_{j=1}^{\nn} \iota^j_{\ast}
      (\sigma_{\qd_j}^m {\sigma}_{\qd_j}^{\prime m} \tau_{\qd_j} {\sigma}_{\qd_j}^{\prime -m}\sigma_{\qd_j}^{-m}\tau_{\qd_j}^{-1} ) \\
      & = \prod_{j=1}^{\nn} \iota^j_{\ast}
      (\sigma_{\qd_j}^m  \tau_{\qd_j} \sigma_{\qd_j}^{-m}\tau_{\qd_j}^{-1} ) \\
      & = \prod_{j=1}^{\nn} \iota^j_{\ast}
      \left( (\sigma_{\qd_j})^m (\tau_{\qd_j} \sigma_{\qd_j}^{-1}\tau_{\qd}^{-1})^m \right) \\
      & = \prod_{j=1}^{\nn} \iota^j_{\ast} ( [\sigma_{\qd_j} , \tau_{\qd_j} ]^m ).
    \end{align*}
Hence, by Lemma \ref{lem:qm}~(2) and Proposition \ref{prop=flux}~(1),
\[
      \mu_P([ \fk_m , \gk ])
       = \sum_{j=1}^{\nn} \mu_P \Bigl(\iota^j_{\ast}([\sigma_{\qd},\tau_{\qd}]^m ) \Bigr)
       = m \sum_{j=1}^{\nn} \mu_P \Bigl(\iota^j_{\ast}([\sigma_{\qd},\tau_{\qd}] ) \Bigr)
       = -m \sum_{j=1}^{\nn} b_jc_j.
\]
 This completes the proof of (1).

 By Lemma \ref{lem=sigmatau}~(3), $\sigma_{\qd_j}$ commutes with $\tau'_{\qd_j}$. By Lemma \ref{lem=sigmatau}~(2), $\sigma'_{\qd_j}$ commutes with ${\tau}^{\prime-1}_{\qd_j} {\sigma}_{\qd_j}^{\prime-1} \tau_{\qd_j}$. Therefore,
\begin{align*}
  [\fk_m,\hk^{-1}] & = \prod_{j=1}^{\nn} \iota^j_{\ast}
  (\sigma_{\qd_j}^m {\sigma}_{\qd_j}^{\prime m} {\tau}^{\prime -1}_{\qd_j} {\sigma}_{\qd_j}^{\prime -m}\sigma_{\qd_j}^{-m}\tau'_{\qd_j} ) \\
  & = \prod_{j=1}^{\nn} \iota^j_{\ast}
  (\sigma_{\qd_j}^m {\sigma}_{\qd_j}^{\prime m} {\tau}^{\prime-1}_{\qd_j} {\sigma}_{\qd_j}^{\prime-m}\tau'_{\qd_j} \sigma_{\qd_j}^{-m}) \\
  & = \prod_{j=1}^{\nn} \iota^j_{\ast}
  \left( \sigma_{\qd_j}^m (\sigma'_{\qd_j})^m ({\tau}^{\prime-1}_{\qd_j} {\sigma}_{\qd_j}^{\prime-1} \tau_{\qd_j}^{\prime} )^m \sigma_{\qd_j}^{-m} \right) \\
  & = \prod_{j=1}^{\nn} \iota^j_{\ast}
  (\sigma_{\qd_j}^m [\sigma'_{\qd_j}, {\tau}^{\prime-1}_{\qd_j} ]^m \sigma_{\qd_j}^{-m}).
\end{align*}
Hence, by   Theorem \ref{thm=mu_P}~(1),  Lemma \ref{lem:qm}~(2) and Proposition \ref{prop=flux}~(2),  we conclude that
\[
      \mu_P([ \fk_m , \hk^{-1} ])
       = \sum_{j=1}^{\nn} \mu_P \Bigl(\iota^j_{\ast}([\sigma'_{\qd_j}, {\tau}^{\prime-1}_{\qd_j} ]^m) \Bigr)
       = m \sum_{j=1}^{\nn} \mu_P \Bigl(\iota^j_{\ast}([\sigma'_{\qd_j}, {\tau}^{\prime-1}_{\qd_j} ] ) \Bigr)
       = -m \sum_{j=1}^{\nn} a_jd_j .
\]
 This completes the proof of (2).

 Finally we prove (3). By (1) and (2), 
\begin{align*}
  \mu_{P}(\gamma_m) & \sim_{D(\mu_P)} \mu_P( [ \fk_m , \gk ] ) + \mu_P( [ \fk_m , \hk^{-1} ]^{-1} ) \\
 & = \mu_P( [ \fk_m , \gk ] ) - \mu_P( [ \fk_m, \hk^{-1} ] ) \\
 & = m\sum_{j=1}^{\nn}(a_id_i-b_ic_i) \\
 & = m \bb (v,w). \qedhere
  \end{align*}
\end{proof}

 Now we are ready to close up the proof of Theorem~\ref{thm=non_ext} (precise version of Theorem~\ref{thm A}).
\begin{proof}[Proof of Theorem~$\ref{thm=non_ext}$]
  Assume that $\mu_{P}$ is extendable to $G_{\Lambda_k}$.
 Take an extension $\mu_P^{\Lambda_k} \colon  G_{\Lambda_k} \to \RR$ of $\mu_P$.
  By Proposition \ref{prop=ext_estimate}, for every $m \in \NN$,
  \[ \mu_P^{\Lambda_k}(\gamma_m) \sim_{3 D(\mu_P^{\Lambda_k})} 0. \]
  Since $\gamma_m \in \Ham(S,\omega)$ for every $m \in \NN$, this means that
  \[ \mu_P(\gamma_m) \sim_{3 D(\mu_P^{\Lambda_k})} 0 \]
  for every $m \in \NN$. By Proposition \ref{prop=intersection}~(3), since $D(\mu_P) \leq D(\mu_P^{\Lambda_k})$,  we have that
  \[ m \bb (v,w) \sim_{4D(\mu_P^{\Lambda_k})} 0;  \]
   hence we obtain that
  \[ \bb (v,w) \sim_{4D(\mu_P^{\Lambda_k})/m } 0  \]
  for every $m \in \NN$.
By letting $m \to \infty$, we obtain that $\bb(v,w)=0$.
This contradicts the assumption that $\bb(v,w)=\frac1k \bb (\bar{v},\bar{w}) \neq 0$.
Therefore, $\mu_{P}$ is not extendable to $G_{\Lambda_k}$.
\end{proof}

\begin{remark} \label{remark tightness 1}
Let $\qd =(a,b,c,d) \in \RR$ satisfy $b=1$, $|c|\leq \ee$ and $|d| \leq \ee$. Set $f_i = \iota^i_{\ast} ( \sigma_{\qd} )$ for $i = 1, \cdots, l$.
Then $f_1, \cdots, f_l$ commute with each other since their supports are disjoint.
By Lemma \ref{local flux calculation}, the image
\[ \flux_\omega(f_1)=[\beta_1]^{\ast} , \cdots, \flux_\omega (f_l)=[\beta_l]^{\ast}   \]
of $f_1, \cdots, f_l$ under $\flux_\omega$ are linearly independent. Therefore, the inequality of Corollary \ref{Zn action} is tight.
\end{remark}

 \begin{remark}\label{remark=normalize}
  If $\mathrm{area}(S)> 1$, then we may modify the proof of Theorem~$\ref{thm=non_ext}$ in the following manner:  scale up $P_{\ee}$ by the factor $\sqrt{\mathrm{area}(S)}$, and replace the embeddings $\iota^1,\ldots,\iota^l$ with embeddings from the scaled $P_{\ee}$ into $S$. This modification provides the better upper bound of $k_0$ in the order of $l\cdot\frac{|\bar{w}|_{\max}}{\mathrm{area}(S)}$. We leave the details of this modification to the reader.
\end{remark}

\subsection{Proof of the Main Theorem }\label{subsection=proof_main}
 Here we establish our main theorem, Theorem~\ref{main thm} (Theorem~\ref{thm=main}).

\begin{proof}[Proof of Theorem~$\ref{main thm}$]
Let $\bar{v}=\flux_{\omega}(f)$ and $\bar{w}=\flux_{\omega}(g)$. Assume that $\bar{v}\smile \bar{w}\ne 0$.
Then, we can apply Theorem~\ref{thm A}, and we obtain $k_0$ in the statement.
Fix $k\in \NN$ with $k\geq k_0$. Let $v=\bar{v}$, $w=\bar{w}/k$, $\Lambda=\Lambda_k=\langle v,w\rangle$ and $G_{\Lambda}=\flux_{\omega}^{-1}(\Lambda)$. Then we have a short exact sequence:
\[
1 \to \Ham(S,\omega) \to G_{\Lambda} \xrightarrow{\flux_{\omega}|_{G_{\Lambda}}} \Lambda \to 1.
\]
Note that $\Lambda_1=\langle \bar{v},\bar{w}\rangle$ is a subgroup of $\Lambda$ of index $k(<\infty)$.
Since $fg=gf$, the map which sends $\bar{v}$ to $f$ and $\bar{w}$ to $g$ gives rise to a group homomorphism $s_1\colon \Lambda_1\to G_{\Lambda}$ with $(\flux_{\omega}|_{G_{\Lambda}}) \circ s_1=\mathrm{id}_{\Lambda_1}$. Hence, the short exact sequence above virtually splits, and Proposition~\ref{prop=discrete} applies.
Then it follows that Py's Calabi quasimorphism $\mu_P$ is extendable to $G_{\Lambda}$, which contradicts Theorem~\ref{thm A}. This completes the proof of Theorem~\ref{main thm}.
\end{proof}

\section{Proof of the $C^0$-version of the Main Theorem}\label{section=C^0}
As a generalization of the (volume) flux homomorphism, Fathi \cite{Fa} considered the \textit{mass flow homomorphism} for measure-preserving homeomorphisms.
Here, for simplicity, we only treat the following special case.

Let $S$ be a closed connected orientable surface
 and $\omega$ a symplectic form on $S$.
Let $\mathrm{Sympeo}(S,\omega)$ denote the $C^0$-closure of $\Symp(S,\omega)$ in the group of homeomorphisms, and
$\mathrm{Sympeo}_0(S,\omega)$ denote the identity component of $\mathrm{Sympeo}(S,\omega)$.
Then, there exists a homomorphism $\tilde\theta_\omega\colon \widetilde{\mathrm{Sympeo}}_0(S,\omega) \to H^1(S;\RR)$, where $\widetilde{\mathrm{Sympeo}}_0(S,\omega)$ is the universal covering of $\mathrm{Sympeo}_0(S,\omega)$.

If the genus of $S$ is at least two, then we have that $\tilde\theta_\omega\left(\pi_1\left(\mathrm{Sympeo}_0(S,\omega)\right)\right)=\{0\}$.
Indeed, since the center of the fundamental group of $S$ is trivial, \cite[Proposition 5.1]{Fa} implies that $\tilde\theta_\omega\left(\pi_1\left(\mathrm{Sympeo}_0(S,\omega)\right)\right)=\{0\}$.
Therefore, when the genus of $S$ is at least two, we obtain the homomorphism $\theta_\omega\colon \mathrm{Sympeo}_0(S,\omega)\to H^1(S;\RR)$ called the \textit{mass flow homomorphism} of $(S,\omega)$.

The mass flow homomorphism satisfies the following properties.
The group $\overline{\Ham(S,\omega)}^{C^0}$ is the $C^0$-closure of the group of $\Ham(S,\omega)$ in the group of homeomorphisms.
\begin{prop}[\cite{Fa}, see also \cite{CGHS}]\label{survey on mfh}
Let $S$ be a closed connected orientable surface whose genus is at least two, and $\omega$ a symplectic form on $S$.
  Then, the following hold true:
 \begin{enumerate}[$(1)$]
  \item $\theta_\omega|_{\Symp_0(S,\omega)}=\flux_\omega$;
  \item $\mathrm{Ker}(\theta_\omega)=\overline{\Ham(S,\omega)}^{C^0}$.
\end{enumerate}
\end{prop}
In particular, we have the following short exact sequence of groups:
\[
1 \to \overline{\Ham(S,\omega)}^{C^0} \to \mathrm{Sympeo}_0(S,\omega) \xrightarrow{\theta_\omega} H^1(S;\RR) \to 1.
\]
\begin{remark}
Despite  the fact that the domain of the mass flow homomorphism in Fathi's original paper looks different from the above one, these two domains coincide.
See Oh and M$\ddot{\mathrm{u}}$ller \cite[Theorem 5.1]{OM} for the proof.
\end{remark}

Here, we restate our goal, Theorem~\ref{C0 main thm}.

\begin{thm}[Restatement of Theorem~$\ref{C0 main thm}$]\label{thm=C^0}
Let $S$ be a closed orientable surface whose genus $l$ is at least two and $\omega$ a symplectic form on $S$. Assume that $f,g \in \Sympeo_0(S,\omega)$ satisfy $fg = gf$. Then the cup product $\theta_\omega(f) \smile \theta_\omega(g)$ equals $0$.
\end{thm}

Before proceeding to the proof of Theorem~\ref{C0 main thm} (Theorem~\ref{thm=C^0}), we outline the strategy of the proof. Recall that in the proof of Theorem~\ref{main thm}, we discuss extendability/non-extendability of Py's Calabi quasimorphism $\mu_P\colon \Ham(S,\omega)\to \RR$.
However, since $\mu_P$ is not $C^0$-continuous (Remark~\ref{remark=disconti}), it is unclear whether $\mu_P$ extends to a quasimorphism on $\overline{\Ham(S,\omega)}^{C^0}$.
The gap lies here in the proof of Theorem~\ref{C0 main thm}.
We address this gap by employing $\mu_{P,B}=\mu_P-\mu_B\colon \Ham(S,\omega)\to \RR$, where $\mu_B$ is the quasimorphism defined by Brandenbursky \cite{Bra}; recall Subsections~\ref{subsec=mu_P_const}, \ref{subsec=mu_P_property} and \ref{subsec=conti}.

\begin{proof}[Proof of Theorem~$\ref{C0 main thm}$]
Set $\mu_{P,B}=\mu_P-\mu_B$, where $\mu_B$ is Brandenbursky's  Calabi  quasimorphism.
By Corollary~\ref{cor=C^0-conti}, we can take a $\mathrm{Sympeo}_0(S,\omega)$-invariant homogeneous quasimorphism $\overline{\mu}_{P,B}$ on $\overline{\Ham(S,\omega)}^{C^0}$ such that $\overline{\mu}_{P,B}|_{\Ham(S,\omega)}=\mu_{P,B}$. By Theorem~\ref{thm=mu_P}~(2), we can take  a quasimorphism $\hat{\mu}_B$ on $\Symp_0(S,\omega)$ such that $\hat{\mu}_B|_{\Ham(S,\omega)}=\mu_B$.

Let $\bar{v}=\theta_{\omega}(f)$ and $\bar{w}=\theta_{\omega}(g)$.
Suppose that $\bar{v}\smile \bar{w}\ne 0$.
Then, we can apply Theorem~\ref{thm A}, and we obtain $k_0$ in the statement of them.
Fix $k\in \NN$ with $k\geq k_0$. Let $v=\bar{v}$, $w=\bar{w}/k$, $\Lambda =\Lambda_k =\langle v,w\rangle$ and $\overline{G}_{\Lambda}=\theta_{\omega}^{-1}(\Lambda)$.
We also set $G_{\Lambda}=\flux_{\omega}^{-1}(\Lambda)$. Then we have the following two short exact sequences:

\begin{center}
  \begin{tikzpicture}[auto]
  \node (01) at (0, 1) {1}; 
  \node (11) at (2, 1) {$\overline{\Ham(S,\omega)}^{C^0}$};
  \node (21) at (4, 1) {$\overline{G}_{\Lambda}$};
  \node (31) at (6, 1) {$\Lambda$};
  \node (41) at (7, 1) {1,};
  \node (00) at (0, 0) {1}; 
  \node (10) at (2, 0) {$\Ham(S,\omega)$};
  \node (20) at (4, 0) {$G_{\Lambda}$};
  \node (30) at (6, 0) {$\Lambda$};
  \node (40) at (7, 0) {1.};
  \draw[->] (01) to node {} (11); 
  \draw[->] (11) to node {} (21);
  \draw[->] (21) to node {{\footnotesize $\theta_{\omega}|_{\overline{G}_{\Lambda}}$}} (31);
  \draw[->] (31) to node {} (41);
  \draw[->] (00) to node {} (10); 
  \draw[->] (10) to node {} (20);
  \draw[->] (20) to node {{\footnotesize $\flux_{\omega}|_{G_{\Lambda}}$}} (30);
  \draw[->] (30) to node {} (40);
  \end{tikzpicture}
\end{center}

Note that $\Lambda_1=\langle \bar{v},\bar{w}\rangle$ is a subgroup of $\Lambda$ of index $k(<\infty)$.
Since $fg=gf$, the map which sends $\bar{v}$ to $f$ and $\bar{w}$ to $g$ gives rise to a group homomorphism $s_1\colon \Lambda_1\to \overline{G}_{\Lambda}$ with $(\theta_{\omega}|_{\overline{G}_{\Lambda}}) \circ s_1=\mathrm{id}_{\Lambda_1}$.
Hence, the first short exact sequence virtually splits. Apply Proposition~\ref{prop=discrete} to the first short exact sequence and $\mu_{P,B}$.
Then, we have a homogeneous quasimorphism $\overline{\mu}_{P,B}^{\Lambda}$ on $\overline{G}_{\Lambda}$ such that $\overline{\mu}_{P,B}^{\Lambda}|_{\overline{\Ham(S,\omega)}^{C^0}}=\overline{\mu}_{P,B}$.
Now we switch our attention from the first short exact sequence to the second one.
Since $G_{\Lambda}$ is a subgroup of $\overline{G}_{\Lambda}$, we can set $\mu_{P,B}^{\Lambda}$ as $\mu_{P,B}^{\Lambda}=\overline{\mu}_{P,B}^{\Lambda}|_{G_{\Lambda}}$.
Then, $\mu_{P,B}^{\Lambda}|_{\Ham(S,\omega)}=\mu_{P,B}(=\mu_P-\mu_B)$ holds. Finally, set $\mu_P^{\Lambda}\colon G_{\Lambda}\to \RR$ by
\[
\mu_P^{\Lambda}=\mu_{P,B}^{\Lambda}+(\hat{\mu}_B|_{G_{\Lambda}}).
\]
Then, $\mu_P^{\Lambda}$ is a homogeneous quasimorphism and $\mu_P^{\Lambda}|_{\Ham(S,\omega)}=\mu_P$.
It follows that $\mu_P$ is extendable to $G_{\Lambda}$, which contradicts Theorem~\ref{thm A}. This completes the proof of Theorem~\ref{C0 main thm}.
\end{proof}

\section*{Acknowledgment}
We truly wish to thank the referees for useful comments and suggestions.
We also thank Professors Masayuki Asaoka and Kaoru Ono, who have drawn the authors' attention to Theorem~\ref{main thm}.
The first author is supported in part by JSPS KAKENHI Grant Number JP18J00765 and 21K13790.
 The second author is supported by JSPS KAKENHI Grant Number JP20H00114 and
 JST-Mirai Program Grant Number JPMJMI22G1.
The third author and the fourth author are supported in part by JSPS KAKENHI Grant Number 19K14536 and 17H04822, respectively.


\bibliographystyle{amsalpha}
\bibliography{KKMM0102}

\end{document}